\documentclass[a4paper,reqno,11pt]{amsart}
\usepackage{amssymb, amsfonts}
\usepackage{fancyhdr, mathrsfs, graphicx, epsfig, amsmath, amsthm}
\usepackage{float}
\usepackage{enumitem}
\usepackage{amscd}
\usepackage{tikz, multicol}
\usepackage[colorlinks,linkcolor=blue,anchorcolor=blue,citecolor=red]{hyperref}
\usetikzlibrary{arrows,decorations.pathmorphing}
\usetikzlibrary{chains}

\theoremstyle{plain}
 \newtheorem{theorem}{Theorem}[section]
 \newtheorem{proposition}[theorem]{Proposition}
 \newtheorem{lemma}[theorem]{Lemma}
 \newtheorem{corollary}[theorem]{Corollary}
 \newtheorem*{assertion*}{Assertion}
 
 \newtheorem{open problem}[theorem]{Open problem}
 \newtheorem{proposed project}[theorem]{Proposed project}
\theoremstyle{definition}
 
 \newtheorem{definition}[theorem]{Definition}
\theoremstyle{remark}
 \newtheorem{remark}[theorem]{Remark}
 \numberwithin{equation}{section}

\title{The MMP singularities of GIT versus Baily-Borel compactifications for the ball quotient case}
\author{Dali Shen}
\date{\today}
\address{Beijing Institute of Mathematical Sciences and Applications, No. 544 Hefangkou, Huairou District, Beijing 101408, China}
\email{shendali@bimsa.cn}

\begin{document}

\setcounter{page}{1}
\pagenumbering{arabic}

\maketitle

\begin{abstract}
When the geometric invariant theory and the Baily-Borel theory both apply to a moduli space, the two resulting 
compactifications do not necessarily coincide. They usually differ by a birational transformation in terms of 
an arrangement, for which the general pattern was firstly described by Looijenga. 
In this paper, we compare the MMP singularities on both sides for the ball quotient case. 
We show that if the relevant arrangement is nonempty, the birational transformation from the GIT compactification 
to the Baily-Borel compactification turns non-log canonical singularities to log canonical singularities. 
We illustrate this with the moduli spaces of quartic curves, rational elliptic surfaces and cubic threefolds.
\end{abstract}

\tableofcontents

\section{Introduction}

This paper intends to compare the Minimal Model Program (MMP) singularities for the two compactifications of a moduli space, 
when the Geometric Invariant Theory (GIT) and the Baily-Borel theory (BB) both applies to that moduli space. 
In the seminal work \cite{Looijenga-2003a,Looijenga-2003b} of Looijenga, it is shown that one can connect two compactifications 
by a birational modification in terms of an arrangement, for both the ball quotient case and the type IV domain quotient case. 
This, later on, leads to the so-called Hassett-Keel-Looijenga program to study the interpolation of the two 
compactifications for a moduli space. In this paper, however, we wish to discuss the MMP singularities on both sides, 
providing another perspective to look at the two algebro-geometric compactifications. We will focus on the ball quotient 
case in the present paper.

Let $\mathbb{B}$ be a complex ball and $\mathbb{L}\rightarrow\mathbb{B}$ be the automorphic line bundle over $\mathbb{B}$. 
Let $\Gamma$ be an arithmetic group acting properly on the pair $(\mathbb{B},\mathbb{L})$ so that we denote the $\Gamma$-orbit 
space by $X:=\mathbb{B}/\Gamma$ and then the automorphic line bundle $\mathbb{L}$ drops to an orbifold line bundle 
$\mathscr{L}$ over $X$. The Baily-Borel theory tells us that $X$ is a quasi-projective variety and it can be completed to 
a projective variety $X^{*}$ by adding finitely many cusps. Suppose we are given a $\Gamma$-invariant locally finite union of 
hyperplane sections in $\mathbb{B}$. This determines a collection of hypersurfaces in $X$ which we assume is also an arrangement 
$\mathscr{H}_{X}$ on $X$, whose complement we denote by $X^{\circ}$. It naturally extends to a collection $\mathscr{H}_{X^{*}}$ 
of hypersurfaces on $X^{*}$ while of which some members may not support a Cartier divisor.

If $X^{\circ}$ can be identified with a space consisting of the stable orbits of an integral normal projective variety with respect to 
a reductive group, then it can also be completed to a compactification in the sense of GIT. To be more specific, let $G$ 
be a reductive group acting on an integral normal projective variety $Y$ endowed with an ample line bundle $\eta$. There exists 
a $G$-invariant open-dense subset $U\subset Y$ consisting of $G$-stable orbits such that $(U,\eta|U)/G$ is a quasi-projective 
variety endowed with an orbifold line bundle. The geometric invariant theory tells us that the algebra of $G$-invariant 
sections of powers of the line bundle $\eta$ is finitely generated and its Proj is a projective variety. The points of 
this variety represent the minimal $G$-orbits in the semistable locus $Y^{ss}\subset Y$, therefore we denote this Proj by 
$Y^{ss}/\!\!/G$.

Suppose that there is an isomorphism between the above two orbifold line bundles 
\[
(U,\eta|U)/G\cong (X^{\circ},\mathscr{L}|X^{\circ}),
\]
assume also a mild boundary condition on both sides (see Theorem $\ref{thm:GIT compactifications}$ for details) for which 
it is usually satisfied when the isomorphism is given by a period map, then it is shown in \cite{Looijenga-2003a} 
that the above isomorphism $U/G\cong X^{\circ}$ extends to an isomorphism 
\[
Y^{ss}/\!\!/G\cong \widehat{X^{\circ}},
\]
where $\widehat{X^{\circ}}$ is a birational modification of $X^{*}$. For that, we first introduce the intermediate Looijenga 
compactification $X^{J}$, a partial resolution of $X^{*}$, so that the collection $\mathscr{H}_{X^{J}}$ of the strict transforms 
of the members in $\mathscr{H}_{X^{*}}$ becomes an arrangement on $X^{J}$. Then the modification is obtained by a sequence of 
blow-ups of $X^{J}$ in terms of the intersection lattice of the arrangement on $X^{J}$ (we denote this resolved space 
by $\tilde{X}$) and then followed by a sequence of blow-downs so that $\widehat{X^{\circ}}$ is still naturally stratified.

Both compactifications $X^{*}$ and $\widehat{X^{\circ}}$ are dominated by $\tilde{X}$ as shown in the following 
diagram.
\[
\begin{tikzpicture}[scale=2]
	\node (A) at (0,0) {$\widehat{X^{\circ}}$};
	\node (B) at (1,1) {$\tilde{X}$};
	\node (C) at (1.5,0.5) {$X^{J}$};
	\node (D) at (2,0) {$X^{*}$};
	\draw
	(A) edge[<-,>=angle 60]      (B)
	(B) edge[->,>=angle 60]      (C)
	(C) edge[->,>=angle 60]      (D)
	(A) edge[<->,>=angle 60,dashed]      (D);
\end{tikzpicture}
\]

Let $\Delta^{*}:=\sum a_{i}D_{i}$ be a $\mathbb{Q}$-divisor on $X^{*}$ where $D_{i}$ is a member of 
the collection $\mathscr{H}_{X^{*}}$ on $X^{*}$ (hence irreducible) and 
\[
a_{i}:=1-\frac{1}{m_{i}}
\]
with $m_{i}$ the ramification order of $D_{i}$. It is not difficult to see that the log pair $(X^{*},\Delta^{*})$ 
is log canonical (lc) since there are only quotient singularities plus cusps on $X^{*}$.

Now the collection $\mathscr{H}_{X^{J}}$ becomes an arrangement on $X^{J}$. Let $\Lambda(\mathscr{H}_{X^{J}})$ denote the 
collection of irreducible components of intersections 
taken from subsets of $\mathscr{H}_{X^{J}}$. For each $L\in\Lambda(\mathscr{H}_{X^{J}})$, it determines an exceptional divisor 
$E(L)$ on $\tilde{X}$ which has a product decomposition $E(L)=\tilde{L}\times\tilde{\mathbb{P}}(L,X^{J})$ and contains an open-dense 
subset $E(L)^{\circ}=L^{\circ}\times\mathbb{P}(L,X^{J})^{\circ}$ where $\mathbb{P}(L,X^{J})$ is the projectivized normal space 
of $L$ in $X^{J}$, and the over-symbol $\tilde{}$ denotes the blow-up with respect to an arrangement and the superscript 
$^{\circ}$ denotes an arrangement complement respectively for a given space if the corresponding arrangement is understood. 
Define 
\[
a_{L}:=(\mathrm{codim}L)^{-1}\sum_{i:D_{i}\supset L}a_{i}.
\]
Now assume $a_{L}-1\neq 0$, 
write a point $z=(z_{0},z_{1})\in E(L)^{\circ}$, let $\tilde{X}_{z}$ denote the germ of $\tilde{X}$ at $z$ and so for $L_{z_{0}}$, 
etc., then there exist a submersion $F_{0}:\tilde{X}_{z}\rightarrow L^{\circ}_{z_{0}}$, 
a submersion $F_{1}:\tilde{X}_{z}\rightarrow T_{1}$ (with $T_{1}$ a vector space, identified with 
$\mathbb{P}(L,X^{J})^{\circ}_{z_{1}}$) and a defining equation $t$ for $E(L)_{z}^{\circ}$ 
such that $(F_{0},t,F_{1})$ is an affine chart for $\tilde{X}_{z}$, and the developing map is projectively equivalent to the map 
\[
[(1,F_{0}),t^{1-a_{L}},t^{1-a_{L}}F_{1}]:\tilde{X}_{z}\rightarrow L^{\circ}_{z_{0}}\times \mathbb{C}\times T_{1}.
\]
The exponent $1-a_{L}$ can be also interpreted as the scalar cone angle of the stratum $L^{\circ}$ in the ball quotient 
$X$ so that it must be positive.

On the other hand, if we regard $E(L)$ as the blow-up of the stratum $Z:=\mathbb{P}(L,X^{J})^{\circ}$ from the GIT 
side $\widehat{X^{\circ}}$, then the above projective developing map can be also written as 
\[
[t^{a_{L}-1}(1,F_{0}),1,F_{1}]:\tilde{X}_{z}\rightarrow L^{\circ}_{z_{0}}\times \mathbb{C}\times T_{1}.
\]
The negativity of the exponent $a_{L}-1$ will then help us to show that the discrepancy of the exceptional divisor $E(Z)$ 
with respect to the variety $\widehat{X^{\circ}}$ is less than $-1$.

Therefore, we have the following theorem.
\begin{theorem}[Theorem $\ref{thm:GIT non-lc}$]
Suppose that the GIT compactification $\widehat{X^{\circ}}$ is a birational modification of the Baily-Borel 
compactification $X^{*}$ in terms of an arrangement, or equivalently, the relevant 
arrangement is nonempty, then the GIT compactification $\widehat{X^{\circ}}$ has non-lc singularities.
\end{theorem}

And it tells us the following fact.
\begin{corollary}[Corollary $\ref{cor:non-lc to lc}$]
The birational transformation $\widehat{X^{\circ}}\dashrightarrow X^{*}$ turns non-lc singularities on $\widehat{X^{\circ}}$ 
to lc singularities on $X^{*}$.
\end{corollary}

We illustrate the above results by discussing three well studied examples to which both GIT and the Baily-Borel theories apply: 
the moduli spaces of quartic curves, rational elliptic surfaces and cubic threefolds.

The paper is organized as follows. In Section $\ref{sec:projective-structure}$, we give a very brief introduction to the 
affine and corresponding projective structures with logarithmic singularities on a complex manifold. In Section 
$\ref{sec:birational-operation}$ we describe the birational operations in terms of the intersection lattice of an 
arrangement, especially for an arrangement on a ball quotient. We then discuss the singularities for ball quotients 
in Section $\ref{sec:singularities}$, mainly from two perspectives: one is the MMP singularities in view of birational 
geometry and the other one is the conical singularities in view of hyperbolic geometry. In Section 
$\ref{sec:GIT singularities}$, we prove the main result for this paper, namely, if the relevant arrangement 
(for the birational transformation) is nonempty, then the GIT compactification $\widehat{X^{\circ}}$ has 
non-lc singularities, thus the birational transformation $\widehat{X^{\circ}}\dashrightarrow X^{*}$ turns non-lc singularities 
on $\widehat{X^{\circ}}$ to lc singularities on $X^{*}$. Finally, in Section $\ref{sec:applications}$, we 
illustrate the above main result with three examples: 
the moduli spaces of quartic curves, rational elliptic surfaces and cubic threefolds.

\vspace{5mm}
\textbf{Acknowledgements.} I thank Klaus Hulek for helpful comments on an earlier version of this paper.

\section{Projective structures}\label{sec:projective-structure}

\subsection{Affine and projective structures}

Let $M$ be a connected complex manifold of complex dimension $n$.

\begin{definition}
A \emph{projective structure} on $M$ is given by an atlas of complex-analytic charts for which the transition maps are 
projective-linear and which is maximal for that property. Likewise, an \emph{affine structure} on $M$ is given by 
an atlas of complex-analytic charts for which the transition maps are affine-linear and which is maximal for that property.
\end{definition}

So a projective structure on $M$ is locally modelled on the pair $(\mathbb{P}^{n}, \mathrm{Aut}(\mathbb{P}^{n}))$ of 
a projective space and its corresponding projective group and an affine structure is locally modelled on the pair 
$(\mathbb{A}^{n}, \mathrm{Aut}(\mathbb{A}^{n}))$ of an affine space and its corresponding affine group.

Recall from \cite{Couwenberg-Heckman-Looijenga} that an affine structure defines a subsheaf $\mathrm{Aff}_{M}$ of 
the structure sheaf $\mathcal{O}_{M}$ consisting of locally affine-linear functions, 
which is of rank $n+1$ and contains the constants. 
The differentials of these make up a local system in the sheaf $\Omega_{M}$ of differentials on $M$,
and such a local system is given by a holomorphic connection on $\Omega_{M}$,
$\nabla: \Omega_{M}\rightarrow\Omega_{M}\otimes\Omega_{M}$ with extension
$\nabla: \Omega_{M}^{r}\otimes\Omega_{M}\rightarrow\Omega_{M}^{r+1}\otimes\Omega_{M}$
by the Leibniz rule
$\nabla(\omega\otimes\zeta)=d\omega\otimes\zeta+(-1)^{r}\omega\wedge\nabla(\zeta)$ for $r\in\mathbb{N}$.
This connection is flat and torsion free. For any connection $\nabla$ on $\Omega_{M}$ its square 
$\nabla^2:\Omega^{r}_{M}\otimes\Omega_{M}\rightarrow\Omega^{r+2}_{M}\otimes\Omega_{M}$
is a morphism of $\mathcal{O}_{M}$-modules, given by wedging with a section $\mathrm{R}$ of
$\mathrm{End}_{\mathcal{O}_{M}}(\Omega_{M},\Omega^{2}_{M}\otimes\Omega_{M})$, called the curvature of $\nabla$,
and $\nabla$ is flat if and only if $\mathrm{R}=0$. The torsion freeness of the connection $\nabla$ on $\Omega_{M}$ 
means that the composite of $\wedge:\Omega_{M}\otimes\Omega_{M}\rightarrow\Omega^{2}_{M}$ with $\nabla:\Omega_{M}\rightarrow\Omega_{M}\otimes\Omega_{M}$
is equal to the exterior derivative $d:\Omega_{M}\rightarrow\Omega^{2}_{M}$.
Indeed flat differentials in $\Omega_{M}$ are then closed,
and by the Poincar\'{e} lemma give rise to a subsheaf of $\mathcal{O}_{M}$ of rank $n+1$ containing the constants.

Conversely, any flat torsion free connection on the cotangent bundle of $M$ defines an affine structure on $M$: 
the subsheaf $\mathrm{Aff}_{M}$ of $\mathcal{O}_{M}$ consisting of holomorphic functions whose total differential 
is flat for $\nabla$ is then of rank $n+1$ containing the constants, 
and the components of the charts in the atlas in question lie in $\mathrm{Aff}_{M}$. 
We refer to Deligne's lecture notes for an excellent exposition of the language of connections 
and more \cite{Deligne-1970}.

Throughout this paper we will not make any notational distinction between a connection on the cotangent bundle and 
the one on any bundle associated to the cotangent bundle (e.g., the tangent bundle).

For a projective structure on $M$, we notice that such a structure defines locally a tautological 
$\mathbb{C}^{\times}$-bundle $\pi: \mathscr{L}^{\times}\rightarrow M$ 
whose total space has an affine structure and for which the scalar multiplication respects the affine structure.
This local $\mathbb{C}^{\times}$-bundle is unique up to scalar multiplication and needs not be globally defined. 
But the flat torsion free connection on the $\mathbb{C}^{\times}$-bundle $\mathscr{L}^{\times}$ is not 
unique, because the way defining it produces not just the tautological 
$\mathbb{C}^{\times}$-bundle, but also a trivialization of the $\mathbb{C}^{\times}$-bundle.

\subsection{Logarithmic singularities}

We will in this subsection give a very brief introduction to a simple type of degeneration along a divisor 
on a complex manifold. A good exposition on this topic is Section $1$ of \cite{Couwenberg-Heckman-Looijenga}. 
A slightly generalized account can be found in Section $3$ of \cite{Shen-2022}.

Let $M$ be a connected complex manifold and $\tilde{M}\rightarrow M$ be a holonomy covering 
and denote its Galois group by $\Gamma$. So $\mathrm{Aff}(\tilde{M}):=
H^{0}(\tilde{M},\mathrm{Aff}_{\tilde{M}})$ is a $\Gamma$-invariant vector space of affine-linear functions on $\tilde{M}$. Then
the set of linear forms $\mathrm{Aff}(\tilde{M})\rightarrow\mathbb{C}$ which are the identity on $\mathbb{C}$ is an affine
$\Gamma$-invariant hyperplane in $\mathrm{Aff}(\tilde{M})^{*}$. We denote this set by $A$.

\begin{definition}
Given a holonomy cover as above, the evaluation map $ev: \tilde{M}\rightarrow A$ which assigns to $\tilde{z}$ the linear form
$ev_{\tilde{z}}\in A: \mathrm{Aff}(\tilde{M})\rightarrow\mathbb{C}; \tilde{f}\mapsto \tilde{f}(\tilde{z})$ is called the
\emph{developing map} of the affine structure.
\end{definition}

The developing map is known to be $\Gamma$-equivariant and a local affine isomorphism. So that it determines 
a natural affine atlas on $M$ whose charts take values in $A$ and whose transition maps lie in $\Gamma$.

\begin{definition}
Suppose an affine structure is given on a complex manifold $M$ by a torsion free flat connection $\nabla$. 
A nowhere zero holomorphic vector field $E$ on $M$ is called a \emph{dilatation field} 
with factor $\lambda\in\mathbb{C}$ if $\nabla_{V}(E)=\lambda V$ for every local vector field $V$.
\end{definition}

If $V$ is flat, then the torsion freeness yields: $[E,V]=\nabla_{E}(V)-\nabla_{V}(E)=-\lambda V$. This tells us that Lie derivative
with respect to $E$ acts on flat vector fields simply as multiplication by $-\lambda$. Hence it acts on flat differentials as
multiplication by $\lambda$.

Let $h$ be a flat Hermitian form 
on the tangent bundle of $M$ such that $h(E,E)$ is nowhere zero. Then the leaf space $M/E$ of the
dimension one foliation induced by $E$ inherits a Hermitian form $h_{M/E}$ in much the same way as the projective space of a finite
dimensional Hilbert space acquiring its Fubini-Study metric. We are especially interested in the case when $h_{M/E}$ is 
of hyperbolic type: 
\[
\lambda\neq 0, \; h(E,E)<0 \; \text{and} \; h>0 \; \text{on} \; E^{\perp}.
\]
Then the leaf space $M/E$ acquires a metric $h_{M/E}$ of constant negative holomorphic sectional curvature, 
for it is locally isometric to a complex hyperbolic space.

In order to understand the behavior of an affine structure near a given smooth subvariety of its singular locus, we need to blow up
that subvariety so that we are led to the situation of codimension one. Let us first look at the simplest degenerating affine
structures, that we will encounter often later on, as follows.

\begin{definition}\label{def:connection}
	Let $D$ be a smooth connected hypersurface in a complex manifold $N$ and an affine structure on $N-D$ is endowed. We say that
	the affine structure on $N-D$ has an \emph{infinitesimally simple degeneration along D} of logarithmic exponents
	$\lambda$ and $\mu\in\mathbb{C}$ if
	\begin{enumerate}
		\item $\nabla$ extends to $\Omega_{N}(\log D)$ with a logarithmic pole along $D$,
		
		\item the residue of this extension along $D$ preserves the subsheaf
		$\Omega_{D}\subset\Omega_{N}(\log D)\otimes\mathcal{O}_{D}$ and its eigenvalue on 
		the quotient sheaf $\mathcal{O}_{D}$ is $\lambda$, and
		
		\item the residue endomorphism restricted to $\Omega_{D}$ is semisimple and all of its eigenvalues are $\lambda$ or $\mu$.
	\end{enumerate}
	The adjective \emph{infinitesimally} can be dropped if in addition 
	\begin{enumerate}[resume]
		\item the connection has a semisimple monodromy on the tangent bundle of $N-D$.
	\end{enumerate}
\end{definition}

If $z$ is a point in $D\subset N$ where $N$ and $D$ are as above, we let $N_{z}$ denote the germ of $N$ at $z$, 
and so for $D_{z}$. We have the following local model for the behavior of the developing map for such a 
degenerating affine structure.

\begin{proposition}
	Let $D$ be a smooth connected hypersurface in a complex manifold $N$ and an affine structure on $N-D$ is endowed. Let $z$ 
	be a point in $D$. Then the infinitesimally simple degeneration along $D$ at $z$ of logarithmic exponents $\lambda,\mu\in\mathbb{C}$
	can also be described in terms of local coordinates.
	
	Namely, there exist a local equation $t$ for $D$ and a local chart
	\[ (F_{\mu},t,F_{\lambda}):N_{z}\rightarrow (T_{\mu}\times\mathbb{C}\times T_{\lambda})_{(0,0,0)} \]
	($T_{\lambda}$ incorporated into $T_{\mu}$ when $\lambda=\mu$), where $T_{\mu}$ and $T_{\lambda}$ are vector spaces, such that the
	developing map near $z$ is affine equivalent to the following multivalued map: 
	\begin{eqnarray*}
		\lambda-\mu\notin\mathbb{Z}&: &(t^{-\mu}F_{\mu},t^{-\lambda},t^{-\lambda}F_{\lambda}),\\
		\lambda-\mu\in\mathbb{Z}_{+}&: &(t^{-\mu}F_{\mu},t^{-\lambda},t^{-\lambda}F_{\lambda})+\log t.t^{-\mu}(0,A\circ F_{\mu}),\\
		&&\text{where $A: T_{\mu}\rightarrow \mathbb{C}\times T_{\lambda}$ is an affine-linear map},\\
		\lambda-\mu\in\mathbb{Z}_{-}&: &(t^{-\mu}F_{\mu},t^{-\lambda},t^{-\lambda}F_{\lambda})+\log t.t^{-\lambda}(B\circ F_{\lambda},0,0),\\
		&&\text{where $B: T_{\lambda}\rightarrow T_{\mu}$ is an affine-linear map},\\
		\lambda=\mu&: &(t^{-\mu}F_{\mu},\log t.t^{-\mu}.\alpha\circ F_{\mu}),\\
		&&\text{where $\alpha: T_{\mu}\rightarrow \mathbb{C}$ with $\alpha(0)\neq 0$ is an affine-linear function}.
	\end{eqnarray*}
	When $\lambda-\mu\notin\mathbb{Z}$, the monodromy around $D_{z}$ is semisimple. For other cases, the monodromy is 
	semisimple if and only if the associated affine-linear map $A$ is zero (for $\lambda-\mu\in\mathbb{Z}_{+}$), 
	$B$ is zero (for $\lambda-\mu\in\mathbb{Z}_{-}$), or $\alpha$ is constant (for $\lambda=\mu$).
\end{proposition}

\begin{proof}
	The same argument as in the proof for Proposition 1.10 of \cite{Couwenberg-Heckman-Looijenga}, as long as 
	we replace the exponents $\lambda$ and $0$ thereof by $\lambda$ and $\mu$ for here.
\end{proof}

\begin{remark}
	We can see from the above local model that for the projectivized developing map 
	(also called the \emph{projective developing map} in what follows) 
	only the relative exponent (i.e., exponent difference) matters. 
	That is, for instance, 
	$(t^{-\mu}F_{\mu},t^{-\lambda},t^{-\lambda}F_{\lambda})=t^{-\mu}(F_{\mu},t^{-(\lambda-\mu)},t^{-(\lambda-\mu)}F_{\lambda})$.
\end{remark}

Let us give a more geometric interpretation for this proposition. Assuming the same as above, 
then there is a canonical decomposition for germs, $D_{z}=D_{z,\mu}\times D_{z,\lambda}$, which is induced by the eigenspaces 
decomposition of the residue endomorphism restricted to the tangent bundle of $D$. The first factor is endowed with a natural 
affine structure and the second factor with a natural projective structure. This shows that the hypersurface $D$ locally 
looks like an exceptional divisor of a blowup which around $z$ is the projection $D_{z}\rightarrow D_{z,\mu}$.

Now we need to understand what happens in case $D$ is a normal crossing divisor 
in the complex manifold $N$ and the affine
structure on $N-D$ degenerates infinitesimally simply along each irreducible component of $D$.

For that let us first look at the case when $D$ has exactly two smooth irreducible components $D_{1}$ and $D_{2}$. 
Denote by $\rho_{1}$ (resp. $\rho_{2}$) the residue endomorphism along $D_{1}$ (resp. $D_{2}$). Let $z$ be a point in 
$K:=D_{1}\cap D_{2}$. Under the assumption of infinitesimally simple degeneration, the local eigenvalue foliations of 
$D_{1}-K$ induced by $\rho_{1}$ and those foliations coming from $D_{2}-K$ induced by $\rho_{2}$ can extend across $K$ in 
a compatible way. For instance, the affine structure of $D_{1,z,\mu_{1}}$ has an infinitesimally simple degeneration along 
$D_{1,z,\mu_{1}}\cap D_{2}$ with exponent $\mu_{2}$ or $\lambda_{2}$. Therefore we have the eigenvalue pairs of 
$(\rho_{1},\rho_{2})$ in $\Omega_{N}(\log D)\otimes \mathbb{C}_{z}$ lying in $\{\mu_{1},\lambda_{1}\}\times \{\mu_{2},\lambda_{2}\}$. 
However, we only need to deal with the special case for which the pair $(\mu_{1},\lambda_{2})$ does not occur. That is because 
for our applications $D$ looks like an exceptional divisor of an iterated blowup, namely, the irreducible components of $D$ 
at $z$ always come with an order: taking our case at hand for example, $D_{2}$ comes after $D_{1}$. In particular, 
$\lambda_{1}-\mu_{1}\neq 0$.

The nonoccurence of $(\mu_{1},\lambda_{2})$ can also be interpreted in a geometric way. Since the affine-linear elements 
in $\mathcal{O}_{N}|_{D_{1}-K}$ are invariant under the holonomy of $D_{2}$, we can define an affine retraction 
$r_{1}: N_{z}\rightarrow D_{1,z,\mu_{1}}$, and likewise $r_{2}: N_{z}\rightarrow D_{2,z,\mu_{2}}$. Note that they commute. 
We can say more if $(\mu_{1},\lambda_{2})$ does not occur: the local affine factor of $D_{1}$ as a quotient of 
its ambient germ remains affine at $z$: namely, $r_{1}r_{2}=r_{1}$.

We can still use the local model to write out the developing map near $z$. There exists a local equation $t_{1}$ (resp. $t_{2}$) for 
$D_{1}$ (resp. $D_{2}$). When $\lambda_{2}-\mu_{2}\neq 0$, There exists a local chart 
$(F_{0},t_{1},F_{1},t_{2},F_{2}):N_{p}\rightarrow (T_{0}\times\mathbb{C}\times T_{1}\times\mathbb{C}\times T_{2})_{(0,0,0,0,0)}$ 
near $z$, where $T_{i}$ is a vector space, such that the developing map is affine-equivalent to 
\begin{multline*}
	t_{1}^{-\mu_{1}}t_{2}^{-\mu_{2}}\big(F_{0},t_{1}^{-(\lambda_{1}-\mu_{1})}(1,F_{1}), 
	t_{1}^{-(\lambda_{1}-\mu_{1})}t_{2}^{-(\lambda_{2}-\mu_{2})}(1,F_{2}) \big): \\
	: N_{z}\rightarrow T_{0}\times(\mathbb{C}\times T_{1})\times(\mathbb{C}\times T_{2}).
\end{multline*}
When $\lambda_{2}-\mu_{2}= 0$, There exists a local chart 
$(F_{0},t_{1},F_{1},t_{2}):N_{p}\rightarrow (T_{0}\times\mathbb{C}\times T_{1}\times\mathbb{C})_{(0,0,0,0)}$ 
near $z$, where $T_{i}$ is a vector space, such that the developing map is affine-equivalent to 
\begin{multline*}
	t_{1}^{-\mu_{1}}t_{2}^{-\mu_{2}}\big(F_{0},t_{1}^{-(\lambda_{1}-\mu_{1})}(1,F_{1},\log t_{2}) \big): 
	N_{z}\rightarrow T_{0}\times(\mathbb{C}\times T_{1}\times\mathbb{C}),
\end{multline*}
where $F_{2}$ (resp. $T_{2}$) is incorporated into $F_{1}$ (resp. $T_{1}$).

The above discussion can be generalized to the following proposition for a normal crossing divisor $D$ in a straightforward manner.

\begin{proposition}
	Let $D$ be a simple normal crossing divisor in a complex manifold $N$, whose irreducible components
	$D_{1},\dots,D_{k}$ ($k\geq 2$) are smooth. An affine structure on $N-D$ is endowed with infinitesimally simply degeneration
	along $D_{i}$ of logarithmic exponents $\mu_{i}$ and $\lambda_{i}$.
	Suppose
	that $\lambda_{i}-\mu_{i}>-1$ and that the holonomy around $D_{i}$ is semisimple unless $\lambda_{i}-\mu_{i}=0$.
	Suppose also that for any pair $(i,j)$, the local affine quotient of the generic point of $D_{i}$ extends across the 
	generic point of $D_{i}\cap D_{j}$.
	
	Let $z$ be a point of $\cap D_{i}$.
	Then $\lambda_{i}-\mu_{i}\neq 0$ for any $i\neq m$ for some $m$ and the
	local affine retraction $r_{i}$ at the generic point of $D_{i}$ extends to $r_{i}: N_{z}\rightarrow D_{i,\mu_{i}}$ in such a manner
	that $r_{i}r_{j}=r_{i}$ for $i<j$.
	
	Furthermore, there exists a local equation $t_{i}$ for each $D_{i}$ and when $\lambda_{i}-\mu_{i}\neq 0$ for any $i$ (resp. 
	$\lambda_{i}-\mu_{i}\neq 0$ for any $i$ but $m$) a morphism to a vector space $F_{i}:N_{z}\rightarrow T_{i}$ 
	with $F_{i}(z)=0$ for $i=0,1,\dots,k$ such that 
	$(F_{0},(t_{i},F_{i})_{i=1}^{k})$ (resp. $(F_{0},(t_{i},F_{i})_{i=1}^{m-1},t_{m},(t_{i},F_{i})_{i=m+1}^{k})$) defines a local 
	chart for $N_{z}$ 
	and the developing map is affine-equivalent to the following multivalued map
	\begin{multline*}
		t_{1}^{-\mu_{1}}\cdots t_{k}^{-\mu_{k}}\Big(F_{0},\big(t_{1}^{-(\lambda_{1}-\mu_{1})}\cdots 
		t_{i}^{-(\lambda_{i}-\mu_{i})}(1,F_{i})\big)_{i=1}^{k} \Big): \\
		: N_{z}\rightarrow T_{0}\times\prod_{i=1}^{k}(\mathbb{C}\times T_{i}) \quad \text{resp.}  \\
		t_{1}^{-\mu_{1}}\cdots t_{k}^{-\mu_{k}}\Big(F_{0},\big(t_{1}^{-(\lambda_{1}-\mu_{1})}\cdots 
		t_{i}^{-(\lambda_{i}-\mu_{i})}(1,F_{i})\big)_{i=1}^{m-1},t_{1}^{-(\lambda_{1}-\mu_{1})}\cdots 
		t_{m-1}^{-(\lambda_{m-1}-\mu_{m-1})}\log t_{m}, \\ 
		,\big(t_{1}^{-(\lambda_{1}-\mu_{1})}\cdots t_{i}^{-(\lambda_{i}-\mu_{i})}(1,F_{i})\big)_{i=m+1}^{k} \Big): \\
		: N_{z}\rightarrow T_{0}\times\prod_{i=1}^{m-1}(\mathbb{C}\times T_{i})\times\mathbb{C}\times \prod_{i=m+1}^{k}(\mathbb{C}\times T_{i}).
	\end{multline*}
\end{proposition}

This proposition makes it possible to perform the compactifiaction defined by an arrangement, 
which will be introduced in the next section. For this operation to be done, we first need to carry out a sequence of iterated blowups
in terms of the \emph{intersection lattice} of the arrangement so as to get a big resolution.
Through this we will arrive at the situation mentioned in the proposition. Then we can contract those exceptional
divisors to get the desired compactification without worrying about the compatibility of their local structures.
We will apply this technique in Section $\ref{sec:GIT singularities}$ when we deal with the projective developing map 
near the given arrangement.

\section{Birational operations in terms of arrangements}\label{sec:birational-operation}

In this section we briefly review the work on the birational relationships between the GIT compactification and Baily-Borel 
compactification in terms of an arrangement for the ball quotients case, which was established in the seminal work of Looijenga 
\cite{Looijenga-2003a}.

Since the mild degenerations of a family of smooth algebraic varieties usually form a codimension $1$ subvariety so as to compactify 
the moduli space, which is closely related to a hyperplane arrangement, let us recall what an arrangement is in various situations.

Let $X$ be a connected complex-analytic manifold, and let $\mathscr{H}$ be a collection of smooth connected hypersurfaes of $X$. 
Then $(X,\mathscr{H})$ is said to be an \emph{arrangement} if the collection $\mathscr{H}$ is locally finite and every point of 
$X$ admits a coordinate neighborhood such that each $H\in\mathscr{H}$ meeting that neighborhood is given by a linear equation.
We define the \emph{arrangement complement} as $X^{\circ}:=X-\cup_{H\in\mathscr{H}}H$, i.e., the complement of the
union of the members of $\mathscr{H}$ in $X$. We will always use the superscript $^{\circ}$ to denote the complement of
an arrangement in an analogous situation as long as the arrangement is understood.

We denote by $\Lambda(\mathscr{H})$ the collection of irreducible components of intersections 
taken from subsets of $\mathscr{H}$; here we understand $X$ is also included in $\Lambda(\mathscr{H})$
as the intersection over the empty subset of $\mathscr{H}$, thus $\Lambda(\mathscr{H})$ forms a partially ordered set. 
And for $L\in\Lambda(\mathscr{H})$, we denote by $\mathscr{H}^{L}$ the collection of $H\in\mathscr{H}$ which passes through $L$. 
On the other hand, each $H\in\mathscr{H}-\mathscr{H}^{L}$ meets $L$ in a hypersurface of $L$. We denote the collection of 
these hypersurfaces of $L$ by $\mathscr{H}_{L}$. But we notice that the natural map 
$\mathscr{H}-\mathscr{H}^{L}\rightarrow\mathscr{H}_{L}$ needs not be injective, so that $\mathscr{H}-\mathscr{H}^{L}$ 
and $\mathscr{H}_{L}$ cannot be identified in general.

Given an arrangement $(X,\mathscr{H})$, if $L\in\Lambda(\mathscr{H})$, 
we want to trivialize the projectivized normal bundle of $L$ in $X$, for which the linearization of an arrangement is expected. 
Suppose that we are given a line bundle $\mathscr{L}$ on $X$ and for every $H\in\mathscr{H}$, an isomorphism 
$\mathscr{L}(H)\otimes\mathscr{O}_{H}\cong\mathscr{O}_{H}$ which also means an isomorphism between the normal bundle $\nu_{H/X}$ and the 
coherent restriction of the dual bundle $\mathscr{L}^{*}$ to $H$. Then the normal bundle $\nu_{L/X}$ is naturally realized 
as a subbundle of $\mathscr{L}^{*}\otimes\mathscr{O}_{L}\otimes_{\mathbb{C}}\mathbb{C}^{\mathscr{H}^{L}}$ 
by the following homomorphism 
\[
\nu_{L/X}\rightarrow\bigoplus_{H\in\mathscr{H}^{L}}\nu_{H/X}\otimes\mathscr{O}_{L}\cong
\mathscr{L}^{*}\otimes\mathscr{O}_{L}\otimes_{\mathbb{C}}\mathbb{C}^{\mathscr{H}^{L}}.
\]
Therefore, a \emph{linearization} of an arrangement $(X,\mathscr{H})$ consists of, as above, the data of a line bundle $\mathscr{L}$ 
on $X$ and for every $H\in\mathscr{H}$, an isomorphism $\mathscr{L}(H)\otimes\mathscr{O}_{H}\cong\mathscr{O}_{H}$ such that, for 
any irreducible component $L$ of an intersection of members of $\mathscr{H}$, the image of the above embedding 
$\nu_{L/X}\rightarrow\mathscr{L}^{*}\otimes\mathscr{O}_{L}\otimes_{\mathbb{C}}\mathbb{C}^{\mathscr{H}^{L}}$ is equal to 
$\mathscr{L}^{*}\otimes N(L,X)$ for some linear subspace $N(L,X)$ of $\mathbb{C}^{\mathscr{H}^{L}}$. We then call 
$N(L,X)$ the \emph{normal space} of $L$ in $X$, and its projectivization $\mathbb{P}(L,X):=\mathbb{P}(N(L,X))$ the \emph{projectively 
normal space} of $L$ in $X$.

There are various arrangements in different settings. A basic example is that when $X$ is a projective space $\mathbb{P}^{n}$, 
$\mathscr{H}$ is a finite collection of projective hyperplanes, and $\mathscr{L}:=\mathscr{O}_{\mathbb{P}^{n}}(-1)$ is the 
tautological line bundle. This is called a \emph{projective arrangement}.

We are also concerned with the arrangement on a ball quotient in this 
paper, which descends from a so-called \emph{complex ball arrangement}: 
Let $W$ be a complex vector space of dimension $n+1$ equipped with a Hermitian form 
$h:W\times W\rightarrow\mathbb{C}$ of signature $(1,n)$, then the complex ball $\mathbb{B}$ is the open subset in $\mathbb{P}(W)$ 
defined by $\mathbb{B}:=\{[z]\in\mathbb{P}(W)\mid h(z,z)>0\}$, which is of complex dimension $n$. Let the line bundle 
$\mathscr{L}:=\mathscr{O}_{\mathbb{B}}(-1)$ be the restriction of the tautological line bundle $\mathscr{O}_{\mathbb{P}(W)}(-1)$ 
to $\mathbb{B}$. This is also the natural automorphic line bundle over $\mathbb{B}$. 
On the other hand, the complex ball $\mathbb{B}$ is also the Hermitian symmetric domain associated to the special unitary group 
$\mathrm{SU}(h)$, and the group $\mathrm{SU}(h)$ acts on the line bundle $\mathscr{O}_{\mathbb{B}}(-1)$ as well.
Every hyperplane $H\subset W$ of hyperbolic signature gives rise to a hyperplane section $\mathbb{B}_{H}:=\mathbb{P}(H)\cap\mathbb{B}$ 
of $\mathbb{B}$. As before, $\mathbb{B}_{H}$ is the Hermitian symmetric domain associated to the special unitary group $\mathrm{SU}(H)$ 
of $H$. With respect to the corresponding hyperbolic metric on $\mathbb{B}$ descended from $h$, $\mathbb{B}_{H}$ is a totally 
geodesic hypersurface of $\mathbb{B}$ and any such hypersurface is of this from. Notice that the normal bundle of $\mathbb{B}_{H}$ 
is naturally isomorphic to $\mathscr{O}_{\mathbb{B}_{H}}(1)\otimes_{\mathbb{C}}W/H$. Hence we have an $\mathrm{SU}(H)$-equivariant 
isomorphism 
\[
\mathscr{O}_{\mathbb{B}}(-1)(\mathbb{B}_{H})\otimes\mathscr{O}_{\mathbb{B}_{H}}\cong
\mathscr{O}_{\mathbb{B}_{H}}\otimes_{\mathbb{C}}W/H\cong\mathscr{O}_{\mathbb{B}_{H}}.
\]
The local finite collection of hyperplane sections of $\mathbb{B}$ is automatically satisfied in an arithmetic setting. 
Suppose that $(W,h)$ is defined over an imaginary quadratic extension $k$ of $\mathbb{Q}$, which is thought of as a subfield 
of $\mathbb{C}$. So we are given a $k$-vector space $W(k)$ equipped with a Hermitian form $h(k):W(k)\times W(k)\rightarrow k$ 
such that after scalar extension we have $(W,h)$. We say that a subgroup of $\mathrm{SU}(h)(k)$ is \emph{arithmetic} if it is 
commensurable with the group $\mathrm{SU}(h)(\mathcal{O}_{k})$ where $\mathcal{O}_{k}$ denotes the ring of integers of $k$. 
It is known that an arithmetic subgroup of $\mathrm{SU}(h)(k)$ acts properly discontinuously on $\mathbb{B}$. We say that 
a collection $\mathscr{H}$ of hyperbolic hyperplanes of $W$ is \emph{arithmetically defined} if $\mathscr{H}$, regarded as a subset 
of $W^{*}$, is a finite union of orbits of an arithmetic subgroup of $\mathrm{SU}(h)(k)$, and each point of $\mathscr{H}$ is 
defined over $k$. In that case the corresponding collection of hyperplane sections $\mathscr{H}|\mathbb{B}$ of $\mathbb{B}$ is 
locally finite.

\subsection{Birational operations on arrangements}
Suppose that $X$ is a variety. Let $(X,\mathscr{H})$ be an arrangement, 
the arrangement $\mathscr{H}$ is in general not simple normal crossing. 
There is, however, a straightforward way to find a birational modification $\tilde{X}^{\mathscr{H}}\rightarrow X$ of $X$ 
such that the preimage of the arrangement is a simple normal crossing divisor. Indeed, the modification $\tilde{X}^{\mathscr{H}}$ 
is obtained by blowing up the members of $\Lambda(\mathscr{H})$ in their natural partial order: 
we first blow up the dimension $0$ members, then the strict transform of the dimension $1$ members, 
and so on, until the strict transform of dimension $n-2$ members: 
\[
X=X_{0}\leftarrow X_{1}\leftarrow\cdots\leftarrow X_{n-2}=\tilde{X}^{\mathscr{H}}.
\]
We refer to $\tilde{X}^{\mathscr{H}}$ as the blow-up of $X$ defined by the arrangement $\mathscr{H}$. We may also write $\tilde{X}$ 
for the corresponding blow-up $\tilde{X}^{\mathscr{H}}$ of $X$ if the arrangement $\mathscr{H}$ on $X$ is understood. For other 
analogous notations which a priori depend on $\mathscr{H}$, we often as well suppress it if the arrangement is understood. 
We will identify $X^{\circ}$ with its preimage in $\tilde{X}$.

If $\mathscr{L}$ is an invertible sheaf on $X$, then we write $\mathscr{L}(\mathscr{H})$ for the subsheaf 
$\sum_{H\in\mathscr{H}}\mathscr{L}(H)$ of the sheaf of rational sections of $\mathscr{L}$. This sheaf is a coherent 
$\mathscr{O}_{X}$-module but need not be invertible. The morphism $f:\tilde{X}\rightarrow X$ is obtained by blowing up 
the fractional ideal sheaf $\mathscr{O}(\mathscr{H})$.

Now every $L\in\Lambda(\mathscr{H})$ determines an exceptional divisor $E(L)$ and all these 
exceptional divisors form a simple normal crossing divisor in $\tilde{X}$. Since $\mathscr{H}_{L}$ defines an arrangement 
in $L$, we have defined $\tilde{L}$ and $L^{\circ}$ in the same manner as above, 
where $\tilde{L}$ denotes the strict transform of $L$ until before blowing up itself. 
Likewise, the members of $\mathscr{H}^{L}$ give rise to a projective arrangement in $\mathbb{P}(L,X)$, 
so we have defined $\tilde{\mathbb{P}}(L,X)$ and $\mathbb{P}(L,X)^{\circ}$ as well. Then the divisor $E(L)$ 
has a product decomposition as follows
\[
E(L)=\tilde{L}\times\tilde{\mathbb{P}}(L,X).
\]
It contains an open-dense subset $E(L)^{\circ}$ identified with 
\[
E(L)^{\circ}=L^{\circ}\times\mathbb{P}(L,X)^{\circ}.
\]

Moreover, the divisors $E(L)$ also determine a stratification of $\tilde{X}$. An arbitrary stratum can be described as follows: 
the intersection $\cap_{i}E(L_{i})$ is nonempty if and only if these members make up a flag: 
$L_{\bullet}:=L_{0}\subset L_{1}\subset \cdots \subset L_{k}\subset L_{k+1}=X$. Their intersection $E(L_{\bullet})$ 
can be identified with the product 
\[
E(L_{\bullet})=\tilde{L}_{0}\times\tilde{\mathbb{P}}(L_{0},L_{1})\times\cdots\times\tilde{\mathbb{P}}(L_{k},X),
\]
and it contains an open-dense stratum $E(L_{\bullet})^{\circ}$ decomposed as follows 
\[
E(L_{\bullet})^{\circ}=L_{0}^{\circ}\times\mathbb{P}(L_{0},L_{1})^{\circ}\times\cdots\times\mathbb{P}(L_{k},X)^{\circ}.
\]

We notice that the above blow-up procedure turns the arrangement into a simple normal crossing divisor, but it is clearly 
not minimal with respect to that property.

Suppose that $X$ is compact, that some positive power of $\mathscr{L}(\mathscr{H})$ is generated by its sections, and that 
these sections separate the points of $X^{\circ}$. Then so are $\tilde{X}$ and the pullback $\pi^{*}\mathscr{L}(\mathscr{H})$ 
on $\tilde{X}$, and the latter is a semiample invertible sheaf. A positive power of this sheaf defines a morphism from $\tilde{X}$ 
to a projective space, which will be constant on the fibers of the projections $E(L)\rightarrow \tilde{\mathbb{P}}(L,X)$. 
If we write $R$ as the equivalence relation on $\tilde{X}$ generated by these projections, then the morphism will factor through 
the quotient space $\tilde{X}/R$. On a stratum $E(L_{0}\subset\cdots\subset L_{k}\subset X)^{\circ}$ we note that the 
equivalence relation is defined by the projection to the last factor $\mathbb{P}(L_{k},X)^{\circ}$. So the quotient space 
$\tilde{X}/R$ is the disjoint union of $X^{\circ}$ and the projective arrangement complements $\mathbb{P}(L,X)^{\circ}$ 
for $L\in\Lambda(\mathscr{H})$, and 
the corresponding projective contraction $\tilde{X}\rightarrow \hat{X}:=\tilde{X}/R$ has the property that $\hat{X}-X^{\circ}$ 
is naturally stratified into subvarieties indexed by $\Lambda(\mathscr{H})$: to $L$ corresponds a copy of 
$\mathbb{P}(L,X)^{\circ}$. This assignment clearly reverses the order relation.

If $Y$ is a normal subvariety of $X$ which meets the members of $\Lambda(\mathscr{H})$ in a transversal way, the strict 
transform of $Y$ in $\hat{X}$ depends only on the restriction of the arrangement to $Y$. This tells us that the smoothness 
condition can be weakened to normality.

Now let $X$ be a normal connected (hence irreducible) variety. A \emph{linearized 
arrangement} on $X$ consists of a locally finite collection of (reduced) Cartier divisors $\{H\}_{H\in\mathscr{H}}$ on $X$, 
a line bundle $\mathscr{L}$, and for each $H\in\mathscr{H}$ an isomorphism 
$\mathscr{L}(H)\otimes\mathscr{O}_{H}\cong\mathscr{O}_{H}$ such that for every irreducible component $L$ of an intersection 
of members of $\mathscr{H}$ the following conditions are satisfied: (i) the natural homomorphism 
$\bigoplus_{H\in\mathscr{H}^{L}}\mathscr{I}_{H}/\mathscr{I}_{H}^{2}\otimes\mathscr{O}_{L}\rightarrow 
\mathscr{I}_{L}/\mathscr{I}_{L}^{2}$ is surjective and not identically zero on any summand; (ii) if we identify 
$\bigoplus_{H\in\mathscr{H}^{L}}\mathscr{I}_{H}/\mathscr{I}_{H}^{2}\otimes\mathscr{O}_{L}$ with 
$\mathscr{L}\otimes\mathscr{O}_{L}\otimes_{\mathbb{C}}\mathbb{C}^{\mathscr{H}^{L}}$ via the given isomorphisms, then 
the kernel of the above homomorphism is spanned by a subspace $K(L,X)\subset\mathbb{C}^{\mathscr{H}^{L}}$ whose codimension 
is that of $L$ in $X$. These conditions imply that the conormal sheaf $\mathscr{I}_{L}/\mathscr{I}_{L}^{2}$ is a locally 
free $\mathscr{O}_{L}$-module of rank equal to the codimension of $L$ in $X$. The normal space $N(L,X)$ is then defined as the 
space of linear forms on $\mathbb{C}^{\mathscr{H}^{L}}$ which vanish on $K(L,X)$. Now with these conditions the 
preceding discussion also holds for $X$ being a normal variety.

\subsection{Complex ball arrangements}
It is clear that the discussion of the preceding subsection does not cover the situation of complex ball arrangements, so we 
will deal with this case in this subsection.

Let $\Gamma$ be an arithmetic subgroup of $\mathrm{SU}(h)(k)$ and assume that $\Gamma$ is \emph{neat}, by which we mean the 
eigenvalues of elements of $\Gamma$ generate a torsion-free multiplicative subgroup of $\mathbb{C}^{\times}$. We here write 
$X$ for the $\Gamma$-orbit space of $\mathbb{B}$. A \emph{cusp} of $\mathbb{B}$ relative to the given $k$-structure is an 
element in the boundary of $\mathbb{B}$ in $\mathbb{P}(W)$ defined over $k$. Namely, a cusp corresponds to an isotropic line 
$I$ in $W$ defined over $k$.

Let $\mathbb{L}^{\times}:=\{w\in W\mid h(w,w)>0\}$. So the obvious projection 
$\mathbb{L}^{\times}\rightarrow\mathbb{B}$ is a principal $\mathbb{C}^{\times}$-bundle. Then the automorphic line bundle 
$\mathbb{L}\rightarrow\mathbb{B}$ is the line bundle associated to the principal $\mathbb{C}^{\times}$-bundle 
$\mathbb{L}^{\times}\rightarrow\mathbb{B}$ together with the representation 
$\mathbb{C}^{\times}\rightarrow \mathrm{GL}_{1}(\mathbb{C}),z\mapsto 1/z$, denoted by 
\[
\mathbb{L}:=(\mathbb{L}^{\times}\times^{\mathbb{C}^{\times}}\mathbb{C})
\rightarrow\mathbb{B}.
\]
A holomorphic form on $\mathbb{B}$ of weight $m$, i.e., a holomorphic section of $\mathbb{L}^{\otimes m}$, can be viewed as 
a holomorphic function on $\mathbb{L}^{\times}$ which is homogeneous of degree $-m$, and vice versa. Since $\Gamma$ is neat, 
$\mathbb{L}$ drops to a line bundle over $X$: 
\[
\mathbb{L}/\Gamma\rightarrow X=\mathbb{B}/\Gamma.
\]

Given an isotropic line $I$ in $W$, the space $\mathbb{P}(W)-\mathbb{P}(I^{\perp})$ is the complement of a projective 
hyperplane in a projective space, hence an affine space with translation group $\mathrm{Hom}(W/I^{\perp},I^{\perp})$. 
Consider the fibration
\[
\mathbb{P}(W)-\mathbb{P}(I^{\perp})\rightarrow
\mathbb{P}(W/I)-\mathbb{P}(I^{\perp}/I)
\]
with fibers affine lines. The base $\mathbb{P}(W/I)-\mathbb{P}(I^{\perp}/I)$ is an affine space over 
$\mathrm{Hom}(W/I^{\perp},I^{\perp}/I)$, denoted by $\mathbb{B}(I)$. The domain $\mathbb{B}$ is fibered over 
$\mathbb{B}(I)$ with fibers half-lines.

We can describe $\mathbb{B}$ as such in terms of (partial) coordinates: if $e$ is a generator of $I$, choose $c\in W$ 
with $h(c,c)=0$ such that $h(e,c)=\sqrt{-1}/2$. The orthogonal complement $A:=\{e,c\}^{\perp}$ of the span of $e$ and 
$c$ is a negative definite inner product space, which can be identified with $I^{\perp}/I$. We then have an isomorphism 
between the two affine spaces $\mathbb{C}\times A$ and $\mathbb{P}(W)-\mathbb{P}(I^{\perp})$, realized by the map 
$(s,a)\in\mathbb{C}\times A\mapsto [se+c+a]\in\mathbb{P}(W)-\mathbb{P}(I^{\perp})$, and in these terms $\mathbb{B}$ is 
defined by the condition $\mathrm{Im}(s)>h(a,a)$. This is the realization of $\mathbb{B}$ as a unbounded Siegel domain 
of the second kind.

Let $J$ be a degenerate subspace of $W$ containing the isotropic line $I$, so that we have $I=J\cap J^{\perp}$ and 
$I^{\perp}=J+J^{\perp}$. Then we write 
\[
\mathbb{L}^{\times}(J):=W/J-I^{\perp}/J \qquad \text{and} \qquad 
\mathbb{B}(J):=\mathbb{P}(W/J)-\mathbb{P}(I^{\perp}/J).
\]
Let $N_{I}$ denote the unipotent radical of the stabilizer group $\mathrm{SU}(h)_{I}$ of $I$. It is known to be a 
Heisenberg group. The group $N_{J}$ of elements of $N_{I}$ that act as the identity on $J^{\perp}$ form a normal 
Heisenberg subgroup of $N_{I}$ for which the quotient $N_{I}/N_{J}$ can be identified with the vector group 
$I^{\perp}/J\otimes\bar{I}$. This vector group may be regarded as the vector space of translations of the affine 
space $\mathbb{B}(J)=\mathbb{P}(W/J)-\mathbb{P}(I^{\perp}/J)$. Let us look at their discrete counterparts, the similar 
statements hold for $\Gamma$ as well. Denote by $\Gamma_{I}$ the stabilizer $\Gamma\cap\mathrm{SU}(h)_{I}$ in 
$\Gamma$ of $I$, let $\Gamma_{J}:=\Gamma_{I}\cap N_{J}$, then $\Gamma_{I}/\Gamma_{J}$ can be identified with a lattice 
in $I^{\perp}/J\otimes\bar{I}$. So if we denote the orbit space of this lattice acting on $\mathbb{B}(J)$ by $X(J)$, 
then $X(J)$ is a principal homogeneous space of a complex torus with universal cover $I^{\perp}/J\otimes\bar{I}$. 
Since $I\subset J\subset I^{\perp}$, we have a natural projection $X(I)\rightarrow X(J)\rightarrow X(I^{\perp})$ of 
abelian torsors.

In order to describe a compactification of $X$ in a uniform way, let us introduce an \emph{intermediate system} which 
assigns to every $k$-isotropic line $I$ a degenerate subspace $J(I)$ defined over $k$ with radical $I$ in a 
$\Gamma$-equivariant manner. The two extremal systems, $BB(I):=I^{\perp}$ and $tor(I):=I$, are called the 
\emph{Baily-Borel system} and the \emph{toroidal system}, respectively.

Such an intermediate system leads to a so-called \emph{Looijenga compactification}. For a given intermediate system 
$\{I\mapsto J(I)\}$ we form the Looijenga partial compactification  
\[
(\mathbb{L}^{\times})^{J}:=\mathbb{L}^{\times}\sqcup\bigsqcup_{\text{$I$ $k$-isotropic}}\mathbb{L}^{\times}(J(I)) 
\quad \text{and} \quad 
\mathbb{B}^{J}:=\mathbb{B}\sqcup\bigsqcup_{\text{$I$ $k$-isotropic}}\mathbb{B}(J(I)), 
\]
endowed with a so-called Satake topology, invariant under $\Gamma$. A basis of this topology on $(\mathbb{L}^{\times})^{J}$ 
is given by (i) the open subsets of $\mathbb{L}^{\times}$ and (ii) for every $k$-isotropic line $I$ the subsets of the form 
$U\sqcup U(J(I))$, where $U$ runs over the open subsets of $\mathbb{L}^{\times}$ invariant under $N_{J(I)}$ and the `positive' 
ray in $I\otimes\bar{I}$. The $\Gamma$-invariant Satake topology on $\mathbb{B}^{J}$ is defined in the same way. 
Notice that the projection 
$(\mathbb{L}^{\times})^{J}\rightarrow\mathbb{B}^{J}$ is a $\mathbb{C}^{\times}$-bundle, so the endowed topology makes 
\[
\mathbb{L}^{J}:=(\mathbb{L}^{\times})^{J}\times^{\mathbb{C}^{\times}}\mathbb{C}
\] 
a topological complex line bundle over $\mathbb{B}^{J}$. We thus get the orbit space 
\[
X^{J}:=\mathbb{B}^{J}/\Gamma
\]
which is compact with respect to the quotient topology, and the line bundle over $X^{J}$ given by 
$\mathbb{L}^{J}/\Gamma$.

Notice that $X^{J}$, as a set, is the disjoint union of $X$ and the torsors $X(J(I))$, where $I$ runs over the representatives 
of the $\Gamma$-orbits in the set of cusps. It is known that $\Gamma$ has only finitely many orbits in its set of cusps, 
and $X^{J}$ is a compact Hausdorff space. The line bundle $\mathbb{L}^{J}\rightarrow\mathbb{B}^{J}$ descends to an analytic 
line bundle on $X^{J}$. Denote its sheaf of sections by $\mathscr{L}$. It is independent of the intermediate system $j$ because 
its restriction to a torsor $X(J(I))$ in the boundary is the trivial line bundle and it realizes the contraction 
$X^{J}\rightarrow X^{*}$ ($X^{*}$ corresponds to the Baily-Borel system). The line bundle $\mathscr{L}$ is ample on $X^{*}$, 
so that the graded algebra 
\[
\bigoplus_{m\geq 0}H^{0}(X^{*},\mathscr{L}^{\otimes m})
\]
is finitely generated and has $X^{*}$ as its Proj.
The two extremal systems give the \emph{Baily-Borel compactification} $X^{*}$ and the \emph{toroidal compactification} $X^{tor}$ 
of $X$ respectively.

Now given a collection $\mathscr{H}$ of hyperbolic hyperplanes of $W$ arithmetically defined relative to $\Gamma$. Note that for 
each $H\in\mathscr{H}$, the hyperplane section $\mathbb{B}_{H}=\mathbb{P}(H)\cap\mathbb{B}$ is a symmetric subdomain 
of $\mathbb{B}$. Denote by $\Gamma_{H}$ its $\Gamma$-stabilizer, which is also arithmetic in its $\mathrm{SU}(h)$-stabilizer. 
So the orbit space $X_{H}:=\mathbb{B}_{H}/\Gamma_{H}$ has its own Baily-Borel compactification $X_{H}^{*}$. 
The inclusion $\mathbb{B}_{H}\subset\mathbb{B}$ induces an analytic morphism $X_{H}\rightarrow X$ whose image is denoted by $D_{H}$, 
and this morphism extends to a finite morphism $X_{H}^{*}\rightarrow X^{*}$ whose image is the closure of $D_{H}$ in $X^{*}$ 
(denoted by $D_{H}^{*}$). 

We assume from now on that the hypersurfaces are not self-intersected in the sense that $X_{H}^{*}\rightarrow D_{H}^{*}$ is a 
homeomorphism and $X_{H}\rightarrow D_{H}$ is an isomorphism. Then the collection $\mathscr{H}_{X^{*}}$ of the hypersurfaces 
$D_{H}^{*}$ is an arrangement on $X$. Note, however, that the hypersurface $D_{H}^{*}$ in $X^{*}$ does not support a Cartier divisor 
in general, while the closure $D_{H}^{tor}$ of $D_{H}$ in $X^{tor}$ is smooth. Thus it is possible to find an intermediate system 
such that the strict transform of $D_{H}^{*}$ on the corresponding intermediate compactification becomes a Cartier divisor. 
To achieve that, let us associate the arrangement 
$\mathscr{H}$ with the assignment $\{I\mapsto J_{\mathscr{H}}(I)=\bigcap_{H\supset I}(H\cap I^{\perp})\}$, 
it is easy to check that it is an intermediate system. And it can be shown that the strict transform of $D_{H}^{*}$ on the 
intermediate compactification $X^{J_{\mathscr{H}}}$, called the Looijenga compactification, becomes a Cartier divisor, 
as desired. In fact it is the smallest normalized blow-up for which it is true.

Then we can apply the discussion of the preceding subsection to $X^{J_{\mathscr{H}}}$ and the corresponding collection 
of the strict transforms of $D_{H}^{*}$, so that we have the following 
birational morphisms and projective completions of $X^{\circ}$
\[
X^{*}\leftarrow X^{J_{\mathscr{H}}}\leftarrow \tilde{X}^{\mathscr{H}}\rightarrow \hat{X}^{\mathscr{H}},
\]
where the morphism $\tilde{X}^{\mathscr{H}}\rightarrow X^{*}$ is the blow-up defined by $\mathscr{O}_{X^{*}}(\mathscr{H}_{X^{*}})$ 
on $X^{*}$. The pullback of the line bundle $\mathscr{L}(\mathscr{H}_{X^{*}})$ to $\tilde{X}^{\mathscr{H}}$ is semiample and 
defines the contraction $\tilde{X}^{\mathscr{H}}\rightarrow \hat{X}^{\mathscr{H}}$. Suppose now that $\mathrm{dim}(\mathbb{B})\geq 2$, 
and that the dimension of a nonempty intersection of members of $\mathscr{H}$ is at least $2$. Then the algebra of 
$\Gamma$-invariant automorphic forms 
\[
\bigoplus_{m\geq 0}H^{0}(\mathbb{B}^{\circ},(\mathbb{L}^{\circ})^{\otimes m})^{\Gamma}
\]
is finitely generated and its Proj is the modification $\hat{X}^{\mathscr{H}}$ of $X^{*}$.

We can apply the above fact to a number of concrete examples of moduli spaces which happen to have a ball quotient structure. 
Let be given an integral normal projective variety $Y$ with an ample line bundle $\eta$ and a reductive group $G$ 
acting on the pair $(Y,\eta)$. 
We are thus given a $G$-invariant open-dense subset $U\subset Y$ consisting of $G$-stable orbits. Then $G$ acts on $U$ properly 
and we have the orbit space $U/G$ as a quasi-projective orbifold with meanwhile the restriction $\eta|U$ descending to an 
orbifold line bundle $(\eta|U)/G$ over $U/G$. Geometric invariant theory tells us that the algebra of $G$-invariants 
$\bigoplus_{m\geq 0}H^{0}(Y,\eta^{\otimes m})^{G}$ is finitely generated. We denote 
its Proj by $Y^{ss}/\!\!/G$ because a point of this Proj can be interpreted as a minimal $G$-orbit in the semistable locus 
$Y^{ss}\subset Y$.

\begin{theorem}[\cite{Looijenga-2003a}]\label{thm:GIT compactifications}
Suppose we are given an identification of $(U,\eta|U)/G$ with a pair coming from a ball arrangement $(X^{\circ},\mathscr{L}|X^{\circ})$, 
such that (i) the boundary $Y^{ss}/\!\!/G-U/G$ is of codimension at least $2$ in $Y^{ss}/\!\!/G$, and (ii) any nonempty intersection of members 
of the arrangement $\mathscr{H}$ with $\mathbb{L}^{\times}$ has dimension at least $2$. Then this identification determines an isomorphism 
\[
\bigoplus_{m\geq 0}H^{0}(Y,\eta^{\otimes m})^{G}\cong 
\bigoplus_{m\geq 0}H^{0}(\mathbb{B}^{\circ},(\mathbb{L}^{\circ})^{\otimes m})^{\Gamma}
\]
for which the algebra of automorphic forms is finitely generated. Moreover, the isomorphism 
$U/G\cong X^{\circ}$ extends to an isomorphism $Y^{ss}/\!\!/G\cong\hat{X}^{\mathscr{H}}$ by taking their Proj's.
\end{theorem}

\begin{remark}
(i) The arrangement $\mathscr{H}$ relevant to $\mathbb{B}^{\circ}$ in the theorem may not be the full arrangement 
associated to $\Gamma$ which endows $X$ with the full orbifold structure.

(ii) The identification demanded by the theorem usually comes from a period map. If so, the (co)dimension conditions in the theorem 
are often satisfied.
\end{remark}

\section{Singularities for ball quotients}\label{sec:singularities}

In this section we discuss what kind of singularities a ball quotient has, mainly from two perspectives: one is the 
MMP singularities in view of birational geometry, and the other one is the conical singularities in view of the hyperbolic geometry.

\subsection{MMP singularities}

There are mainly two types of MMP singularities on a compactification of a ball quotient, one comes from the inside of a ball quotient, 
and the other one comes from the boundary of the ball quotient.

\begin{definition}
A \emph{quotient singularity} is an analytic space germ $(X,q)$ such that it is isomorphic to $(\mathbb{C}^{n}/G,\pi(0))$, 
where $G$ is a finite subgroup of $\mathrm{GL}_{n}(\mathbb{C})$ and $\pi:\mathbb{C}^{n}\rightarrow\mathbb{C}^{n}/G$ is 
the natural projection.
\end{definition}

Let $Y$ be a complex manifold and $\mathrm{Aut}(Y)$ the holomorphic automorphism group of $Y$. If $\Gamma$ is a discrete subgroup 
of $\mathrm{Aut}(Y)$, which implies that $\Gamma$ acts properly discontinuously on $Y$, then each point of the complex quotient 
space $Y/\Gamma$ is a quotient singularity since the isotropy group $\Gamma_{p}$ is finite for any $p\in Y$. Note that we can 
also regard the regular points as trivial singularities.

\begin{definition}
Let $Y$ be a complex manifold. An element $\sigma\in\mathrm{Aut}(Y)$ is called a complex reflection at $p$ if it is of finite order 
and the codimension $1$ stratum of its fixed points set passes through $p$. An element $\sigma$ is called a complex reflection of 
$Y$ if its fixed points set is nonempty and it is a complex reflection at each point of the codimension $1$ stratum of 
its fixed points set.
\end{definition}

\begin{remark}
From the definition we see that if $\sigma$ is a complex reflection at $p$, then its representation $(d\sigma)_{p}$ in the tangent 
space $T_{p}Y$ has exactly one eigenvalue other than $1$ and with multiplicity $1$.
\end{remark}

Now let us see what role a complex reflection plays in a quotient singularity.

\begin{proposition}[\cite{Gottschling}]\label{prop:reflection}
Let $Y$ be a complex manifold, $p\in Y$, and $\Gamma$ a discrete group acting on $Y$ with $\pi: Y\rightarrow Y/\Gamma=X$ the 
natural projection. Then $\pi(p)$ is regular on $X$ if and only if $\Gamma_{p}$ is generated by complex	reflections at $p$.
\end{proposition}

Note also a global result, which for real (involutive) reflections has its origin in Chevalley's article \cite{Chevalley}.
A complex analogue is as follows:

\begin{proposition}[{\cite[V, \S 5.3, Theorem 3]{Bourbaki}}]
If $G$ is a finite subgroup of $\mathrm{GL}_{n}(\mathbb{C})$ generated by complex reflections, then there is
an (algebraic) isomorphism $\mathbb{C}^{n}/G\cong\mathbb{C}^{n}$.
\end{proposition}

Now let $Y$ be a complex ball $\mathbb{B}$ and $\Gamma$ be an arithmetic subgroup of $\mathrm{SU}(h)(k)$, which implies that 
$\Gamma$ acts properly discontinuously on $\mathbb{B}$. So all the stabilizers are finite and the orbit space $X=\mathbb{B}/\Gamma$ 
has quotient singularities only. It implies that $X$ is $\mathbb{Q}$-factorial with rational singularities by the following 
proposition.

\begin{proposition}[{\cite[Proposition 5.15]{Kollar-Mori}}]
Let $X$ be an algebraic or analytic variety over $\mathbb{C}$ with quotient singularities only. Then $X$ has rational singularities 
and $X$ is $\mathbb{Q}$-factorial.
\end{proposition}

Furthermore, $X$ has Kawamata log terminal (klt) singularities. For that we first give a lemma.

\begin{lemma}[{\cite[Proposition 5.20]{Kollar-Mori}}]
Let $f:X'\rightarrow X$ be a finite morphism between $n$-dimensional normal varieties. Let $\Delta=\sum a_{i}D_{i}$ be a 
$\mathbb{Q}$-divisor on $X$ and $\Delta'=\sum a_{j}'D_{j}'$ a $\mathbb{Q}$-divisor on $X'$ such that 
$K_{X'}+\Delta'=f^{*}(K_{X}+\Delta)$. Then $(X,\Delta)$ is klt if and only if $(X',\Delta')$ is klt.
\end{lemma}

\begin{theorem}
$X$ has klt singularities.
\end{theorem}

\begin{proof}
For our $X=\mathbb{B}/\Gamma$, we can always choose a neat sublattice $\Gamma'$ in $\Gamma$ of finite index such that 
$f:X'=\mathbb{B}/\Gamma'\rightarrow X$ is a finite morphism and $X'$ has klt singularities (since $X'$ is smooth). 
Then it follows from the above lemma that $X$ has klt singularities as well.
\end{proof}

Let us next consider the Baily-Borel compactification $X^{*}$ of $X$. Assume that the hypersurfaces $D_{H}^{*}$ are not 
self-intersected, the collection $\mathscr{H}_{X^{*}}$ of the 
hypersurfaces $D_{H}^{*}$ forms an arrangement on $X$, which while in general does not support a Cartier divisor on its 
Baily-Borel compactification $X^{*}$. 
Since the isotropy group at a generic point of $H\in\mathscr{H}$ fixes a codimension $1$ stratum, 
it is generated by a single complex reflection. Index each member $D_{H}^{*}$ of $\mathscr{H}_{X^{*}}$ (hence irreducible) 
simply by $D_{i}$, and denote the ramification order on a 
generic point of each $D_{i}$ by $m_{i}$. Write 
\[
a_{i}:=1-\frac{1}{m_{i}}.
\]
Let $\Delta^{*}=\sum a_{i}D_{i}$ be a $\mathbb{Q}$-divisor on $X^{*}$. 
It is clear that $X^{*}$ is normal, and $K_{X^{*}}+\Delta^{*}$ is $\mathbb{Q}$-Cartier by a result of Alexeev 
\cite[\S 3]{Alexeev}, based on earlier work of Mumford \cite{Mumford}. 
We then have such a log pair $(X^{*},\Delta^{*})$. 
We next show that $(X^{*},\Delta^{*})$ is lc.

\begin{theorem}
$(X^{*},\Delta^{*})$ is lc.
\end{theorem}

\begin{proof}
We denote the restriction of $\Delta^{*}$ on $X$ by $\Delta$, then $(X,\Delta)$ is klt by the above theorem. 
So we need only consider the singularity at the cusp. For that, consider a partial resolution $f:X^{tor}\rightarrow X^{*}$ 
for which the toroidal system associate the isotropic line $I$ a degenerate subspace $J(I)=I$, so that $X(I)$ is an abelian 
hypertorsor. Then the corresponding ramification order $m$ of this divisor goes to $\infty$ since it comes from the cusp. 
Thus the discrepancy of the $f$-exceptional divisor $X(I)$ is 
\[
a(X(I),X^{*},\Delta^{*})=\lim_{m\to\infty}-\frac{m-1}{m}=-1.
\] 
Therefore, $(X^{*},\Delta^{*})$ is lc.
\end{proof}

\subsection{Logarithmic exponents}

Given an arrangement $(X,\mathscr{H})$, the projective structure on $X^{\circ}$ can be obtained from an 
affine structure on its tautological $\mathbb{C}^{\times}$-bundle $\pi:\mathscr{X}^{\circ}\rightarrow X^{\circ}$ 
according to the discussion of Section $\ref{sec:projective-structure}$. 
This affine structure determines a flat torsion free connection $\nabla$ on the (co)tangent bundle of $\mathscr{X}^{\circ}$, 
and has an infinitesimally simple degeneration at a generic point of the lift $\pi^{-1}(D_{i})$ of each $D_{i}$ of 
logarithmic exponents $0$ and $a_{i}$, where we still use $\pi$ to denote the projection of the corresponding completion 
$\pi:\overline{\mathscr{X}^{\circ}}\rightarrow \overline{X^{\circ}}$ which 
respects the scalar multiplication. 
We want to determine the logarithmic exponents of the lift $\pi^{-1}(E(L))$ of the exceptional divisor $E(L)$ for every 
$L\in\Lambda(\mathscr{H})$. 

Let us first have a look at what this kind of connection is like. 
Let $M$ be a connected complex manifold. The connection given in Definition $\ref{def:connection}$ defines a 
special class of holomorphic differentials on $M$ while of logarithmic forms since it extends to a section of 
$\Omega_{\overline{M}}(\log D)$ if $\overline{M}\supset M$ is a smooth completion of $M$ which adds to $M$ a 
normal crossing divisor $D$. This property is independent of the choice of the completion, and according to 
Deligne \cite{Deligne-1971,Deligne-1974} such a differential is always closed. 
If instead of being normal crossing, the divisor $D$ is \emph{arrangement-like}, that is, in local-analytic coordinates given 
by a product of linear forms $l_{1},\dots,l_{k}$, then a regular differential on $M$ is of a logarithmic form 
if and only if in any such coordinate chart it is locally a linear combination of the logarithmic forms 
$\frac{dl_{1}}{l_{1}},\dots,\frac{dl_{k}}{l_{k}}$ with analytic coefficients.

We call the connection $\nabla$ on the trivial vector bundle over $M$ with fiber $F$ a \emph{logarithmic connection} if its 
coefficients are. We denote its connection form by $\Omega\in H^{0}(M,\Omega_{\overline{M}})\otimes_{\mathbb{C}}\mathrm{End}(F)$. 
For such connections we have the following flatness criterion (also can be found in \cite{Looijenga-1999}).

\begin{lemma}\label{lem:flatness}
Let $M$ be a connected complex manifold. Suppose that $\overline{M}\supset M$ is a smooth
completion of $M$ which adds to $M$ an arrangement-like divisor $D$ whose irreducible
components $D_{i}$ are smooth. Suppose that $\overline{M}$ has no nonzero
regular $2$-forms and that any irreducible component $D_{i}$ has no nonzero regular
$1$-forms. Then a logarithmic connection $\nabla$ on the trivial vector bundle $M\times F$ over $M$ is
flat if and only if for every intersection
$I$ of two distinct irreducible components of $D$, the sum
$\sum_{D_{i}\supset I}\mathrm{Res}_{D_{i}}\nabla$ commutes with each of its terms $\mathrm{Res}_{D_{i}}\nabla$
($D_{i}\supset I$).
\end{lemma}

\begin{proof}
First we prove the following fact.

\begin{assertion*}
The condition that $[\sum_{D_{i}\supset I}\mathrm{Res}_{D_{i}}\nabla,\mathrm{Res}_{D_{i}}\nabla]=0$ 
is equivalent to $\mathrm{Res}_{I}\mathrm{Res}_{D_{i}}\mathrm{R}(\nabla)=0$.
\end{assertion*}

The connection form $\Omega$ on $M$ could be locally written as
\[
\Omega=\sum_{i}f_{i}\frac{dl_{i}}{l_{i}}\otimes \rho_{i}+\sum_{i}\omega_{i}\otimes \rho_{i}',
\]
where $f_{i}$'s are holomorphic functions, $D_{i}$ is given by $l_{i}=0$,
$\omega_{i}$'s are holomorphic $1$-forms and $\rho_{i}$'s, $\rho_{i}'$'s are the endomorphisms of $F$.
Then we have
\begin{align*}
		\mathrm{Res}_{D_{i}}\Omega=f_{i}\vert_{l_{i}=0}\rho_{i}
	\end{align*}
	and
	\begin{align*}
		\Omega\wedge \Omega
		=&\sum_{i,j}f_{i}f_{j}\frac{dl_{i}}{l_{i}}
		\frac{dl_{j}}{l_{j}}\otimes \rho_{i}\rho_{j}+\sum_{i,j}f_{i}\frac{dl_{i}}{l_{i}}\wedge\omega_{j}\otimes
		(\rho_{i}\rho_{j}'-\rho_{j}'\rho_{i})\\
		&+\sum_{i,j}\omega_{i}\omega_{j}\otimes \rho_{i}'\rho_{j}'\\
		=&\sum_{i<j}f_{i}f_{j}\frac{dl_{i}}{l_{i}}
		\frac{dl_{j}}{l_{j}}\otimes(\rho_{i}\rho_{j}-\rho_{j}\rho_{i})+\sum\limits_{i,j}f_{i}\frac{dl_{i}}{l_{i}}\wedge\omega_{j}\otimes
		(\rho_{i}\rho_{j}'-\rho_{j}'\rho_{i})\\
		&+\sum\limits_{i,j}\omega_{i}\omega_{j}\otimes \rho_{i}'\rho_{j}'.
	\end{align*}
	Then
	\begin{align*}
		\mathrm{Res}_{D_{i}}\Omega\wedge \Omega=\sum_{j:j\neq i}f_{i}f_{j}
		\frac{dl_{j}}{l_{j}}\vert_{l_{i}=0}\otimes(\rho_{i}\rho_{j}-\rho_{j}\rho_{i})+
		\sum_{j:j\neq i}f_{i}\omega_{j}\otimes [\rho_{i},\rho_{j}'],
	\end{align*}
	As $I\subset D_{i}$ is given by $D_{j}\cap D_{i}$ for any $j\neq i$ with $D_{j}\supset I$,
	we have
	\begin{multline*}
		\mathrm{Res}_{I}\mathrm{Res}_{D_{i}}\Omega\wedge \Omega=\sum\limits_{j: D_{j}\supset I, j\neq i}f_{i}f_{j}\vert
		_{I}(\rho_{i}\rho_{j}-\rho_{j}\rho_{i})=\\
		\sum\limits_{j: D_{j}\supset I, j\neq i}f_{i}f_{j}\vert
		_{I}[\rho_{i},\rho_{j}]
		=[f_{i}\vert_{I}\rho_{i},\sum\limits_{j}f_{j}\vert_{I}\rho_{j}]
		=[\mathrm{Res}_{D_{i}}\Omega,\sum \mathrm{Res}_{D_{j}}\Omega].
\end{multline*}
Since the double residue of $d\Omega$ is obviously zero (any term of $d\Omega$ is of at most
simple pole), the assertion follows.
	
Let us continue to prove the lemma. Necessity is obvious, but it is also sufficient:
If the double residue of $\mathrm{R}(\nabla)$ is equal to zero, then
$\mathrm{Res}_{D_{i}}\mathrm{R}(\nabla)$
has no pole along $I\subset D_{j}\cap D_{i} \; \text{for} \; \forall D_{j}\neq D_{i}$ , hence
$\mathrm{Res}_{D_{i}}\mathrm{R}(\nabla)$ has as coefficients regular $1$-form along $D_{i}$, but there is no nonzero regular
$1$-form along $D_{i}$, we then have $\mathrm{Res}_{D_{i}}\mathrm{R}(\nabla)=0$. Again, $\mathrm{R}(\nabla)$
has no pole along $D_{i}$, hence
$\mathrm{R}(\nabla)$ has as coefficients regular $2$-form everywhere, but there is no nonzero regular
$2$-form on $\overline{M}$, we then have $\mathrm{R}(\nabla)=0$.
\end{proof}

\begin{remark}
Clearly the connection given in Definition $\ref{def:connection}$ is of this type.
\end{remark}

Now consider the tautological $\mathbb{C}^{\times}$-bundle $\pi:\mathscr{X}^{J_{\mathscr{H}}}\rightarrow X^{J_{\mathscr{H}}}$ 
for the intermediate Looijenga compactification of our ball quotient $X=\mathbb{B}/\Gamma$. 
For simplicity we write $X^{J}$ for $X^{J_{\mathscr{H}}}$ in what follows. 
The corresponding collection of the strict transforms of $D_{i}$ forms an arrangement on $X^{J}$, denoted by $\mathscr{H}_{X^{J}}$. 
By abuse of notation, we still denote the strict transform of $D_{i}$ in $X^{J}$ by $D_{i}$. 
The arrangement $\mathscr{H}_{X^{J}}$ on $X^{J}$ is lifted to an arrangement 
$\mathscr{H}_{\mathscr{X}^{J}}$ on $\mathscr{X}^{J}$, 
we denote the lift of each irreducible member $D_{i}$ in $\mathscr{H}_{X^{J}}$ by $\mathscr{D}_{i}$, 
and the lift of each $L\in\Lambda(\mathscr{H}_{X^{J}})$ by $\mathscr{L}$. So there also defines 
an affine structure on $\mathscr{X}^{\circ}$ with infinitesimally simple degeneration at a generic point of $\mathscr{D}_{i}$ 
of logarithmic exponents $0$ with multiplicity $n$ and $a_{i}$ with multiplicity $1$. We want to determine the 
logarithmic exponents of the lift $\pi^{-1}(E(L))$ of the exceptional divisor $E(L)$ for each $L\in\Lambda(\mathscr{H}_{X^{J}})$. 
Let $\nabla$ be such a torsion free flat connection defined on the tangent bundle of $\mathscr{X}^{\circ}$ that determines 
the above affine structure. Denote the generic fiber of the tangent bundle of $\mathscr{X}^{\circ}$ by $F$. Let 
$\rho_{\mathscr{L}}\in\mathrm{End}(F)$ denote the projection with kernel the generic fiber of the tangent bundle of 
$\mathscr{L}$ and image the normal space $N(\mathscr{L},\mathscr{X}^{J})$. 
So the residue of the connection $\nabla$ along each $\mathscr{D}_{i}$ is written as $a_{i}\rho_{\mathscr{D}_{i}}$.

\begin{lemma}
Let $\nabla$ be a torsion free flat logarithmic connection as above with residues $a_{i}\rho_{\mathscr{D}_{i}}$ along each 
$\mathscr{D}_{i}$, and let $\mathscr{L}=\pi^{-1}(L)$ for each $L\in\Lambda(\mathscr{H}_{X^{J}})$. Then the transformation 
$\sum_{\mathscr{D}_{i}\supset\mathscr{L}}a_{i}\rho_{\mathscr{D}_{i}}$ is of the form $a_{\mathscr{L}}\rho_{\mathscr{L}}$, 
where 
\[
a_{\mathscr{L}}=\frac{1}{\mathrm{codim}(\mathscr{L})}\sum_{i:\mathscr{D}_{i}\supset\mathscr{L}}a_{i}.
\]
\end{lemma}

\begin{proof}
It is clear that the sum $\sum_{\mathscr{D}_{i}\supset\mathscr{L}}a_{i}\rho_{\mathscr{D}_{i}}$ is zero on the generic tangent 
space of $\mathscr{L}$ and preserves $N(\mathscr{L},\mathscr{X}^{J})$. Since this sum commutes with each of its terms 
(a direct consequence of Lemma $\ref{lem:flatness}$), it will 
preserve $\mathscr{D}_{i}$ and $N(\mathscr{D}_{i},\mathscr{X}^{J})$ for each $\mathscr{D}_{i}\supset\mathscr{L}$. 
Since the collection $\{\mathscr{D}_{i}\mid\mathscr{D}_{i}\supset\mathscr{L}\}$ contains $\mathrm{codim}(\mathscr{L})+1$ 
members of which each $\mathrm{codim}(\mathscr{L})$-element subset is in general position, the induced transformation 
on $N(\mathscr{L},\mathscr{X}^{J})$ will be scalar. And this scalar operator must have the same trace as the sum 
$\sum_{\mathscr{D}_{i}\supset\mathscr{L}}a_{i}\rho_{\mathscr{D}_{i}}$, so the scalar equals to the number 
$a_{\mathscr{L}}$ as above. Since $N(\mathscr{L},\mathscr{X}^{J})$ is the span of the spaces $N(\mathscr{D}_{i},\mathscr{X}^{J})$ 
for which $\mathscr{D}_{i}\supset\mathscr{L}$, the transformation 
$\sum_{\mathscr{D}_{i}\supset\mathscr{L}}a_{i}\rho_{\mathscr{D}_{i}}$ is of the form $a_{\mathscr{L}}\rho_{\mathscr{L}}$.
\end{proof}

\begin{remark}\label{rem:exponents}
For the projectivized space $X^{J}$, if we write 
\[
a_{L}=\frac{1}{\mathrm{codim}(L)}\sum_{i:D_{i}\supset L}a_{i},
\]
since $\mathrm{codim}(L)=\mathrm{codim}(\mathscr{L})$, and $D_{i}$ contains $L$ if and only if $\mathscr{D}_{i}$ 
contains $\mathscr{L}$, we have that 
\[
a_{L}=a_{\mathscr{L}}.
\]
\end{remark}

\begin{lemma}\label{lem:exponents}
Given an irreducible intersection $\mathscr{L}\in\Lambda(\mathscr{H}_{\mathscr{X}^{J}})$ and denote by $\mathscr{E}$ 
the exceptional divisor of the blow-up of $\mathscr{L}$ in $\mathscr{X}^{J}$. Then the affine structure on $\mathscr{X}^{\circ}$ 
is of infinitesimally simple type along $\mathscr{E}^{\circ}$ (in the sense of Definition $\ref{def:connection}$) 
with logarithmic exponents $0$ and $a_{\mathscr{L}}-1$.
\end{lemma}

\begin{proof}
The same argument as in Lemma 2.21 of \cite{Couwenberg-Heckman-Looijenga}.
\end{proof}

\subsection{Conical singularities}

We discuss the conical singularities for a ball quotient in view of the hyperbolic geometry in this subsection.

We know that a ball quotient $X=\mathbb{B}/\Gamma$ is an orbifold. It actually also can be viewed as a special case 
(with isotropy group being finite) of a more general structure, namely the structure of a cone-manifold \cite{Shen-2021}.
Roughly speaking, a cone-manifold is a manifold endowed with a particular kind of singular Riemannian metric. 

Let $Y$ be a complete connected $n$-dimensional Riemannian manifold and $G$ be a group of isometries of $Y$. Then a 
\emph{$(Y,G)-$manifold} is a space $M$ equipped with an atlas of charts with homeomorphism into $Y$ such that the 
transition maps lie in $G$. Let $G_{p}$ denote the stabilizer of a point $p$ and $Y_{p}$ the unit tangent sphere of $T_{p}Y$. 
Then $(Y_{p},G_{p})$ is a model space of one dimension lower.

Following \cite{Thurston}, a \emph{$(Y,G)-$cone-manifold} can be defined inductively by dimension as follows. 
If $\dim Y=1$, a $(Y,G)-$cone-manifold is just a $(Y,G)-$manifold. If $\dim Y>1$, a $(Y,G)-$cone-manifold is a space 
$M$ such that each point $x\in M$ has a neighborhood modelled, using the geometry of $Y$, on the cone over a compact and 
connected $(Y_{p},G_{p})-$cone-manifold $S_{x}M$, called the \emph{unit tangent cone} to $M$ at $x$.

The unit tangent cone is a special case of a \emph{spherical cone-manifold}, which is an 
$(S^{n},\mathrm{O}(n+1))-$cone-manifold. Let $K$ and $N$ be two compact spherical cone-manifolds. Assuming that each of 
them is connected unless it is isomorphic to $S^{0}$, 
we can define the \emph{join} of $K$ and $N$, denoted by $K*N$, as a spherical cone-manifold obtained from the disjoint 
union of $K$ and $N$ by adding a spherical arc $[a,b]$ of length $\pi/2$ between every pair of points $a\in K$ and $b\in N$. 
For instance, $S^{m}*S^{n}\cong S^{m+n+1}$. The unit tangent cones of the join are computed as follows
\begin{equation*}
S_{x}(K*N)\cong
\left\{
\begin{aligned}
&(S_{x}K)*N  && \text{if $x\in K$,} \\
&K*(S_{x}N)  && \text{if $x\in N$, and} \\
&S^{0}*(S_{a}K)*(S_{b}N)  && \text{if $x$ lies in the open arc $(a,b)$}.
\end{aligned}
\right.
\end{equation*}

We say that $M$ is \emph{prime} if it can not be expressed by a join $M\cong S^{0}\ast N$. Then we have the following 
factorization theorem.
\begin{theorem}[{\cite[Theorem 5.1]{McMullen}}]
A compact spherical cone-manifold $M$ can be canonically expressed as the join $M\cong S^{k}* N$ of a sphere 
and a prime cone-manifold $N$.
\end{theorem}

Given a $(Y,G)-$cone-manifold $M$ of real dimension $n$. A local model for $M$ can be constructed as follows. 
Let $U$ be a neighborhood of $x$ in $M$. Locally, there is an exponential map from a small ball $B^{n}$ in $T_{p}Y$ to 
a neighborhood of $p$ in $Y$
\[
\exp: B^{n}\rightarrow Y,
\]
such that the pull back of the metric on $Y$ endows the small ball $B^{n}$ with a metric $g$, and the metric cone 
$(C_{g}(K),0)$ over $K$, where $K\cong S_{x}M$, is isometric to $(U,x)$. Then the metric cone $(C_{g}(K),0)$ provides an 
isometric local model for $(U,x)$.

There is a natural stratification for $M$ induced by this singular metric. 
Its $k$-dimensional strata are the connected components of the locus characterized by 
\[
M[k]=\{x\in M\mid S_{x}M\cong S^{k-1}*N, \text{with $N$ prime}\}.
\]
We refer to the prime factor $N$ above as the \emph{normal cone} to $M[k]$ at $x$, which we write as 
$N_{x}M$ in what follows. Thus the above decomposition can be rewritten as 
\[
S_{x}M\cong S^{k-1}*N_{x}M
\]
for all $x\in M[k]$.

We have the local model $(C_{g}(S_{x}M),0)$ for a neighborhood $U$ of $x\in M[k]$, it is not hard to see that 
\[
M[k]\cap U\cong C_{g}(S^{k-1}).
\]
This also shows that the geodesic ray from $x$ to any nearby point $y\in M[k]$ entirely lies in $M[k]$ so that 
$M[k]$ is totally geodesic in $M$. 
It follows that the loci $M[k]$ ($0\leq k\leq n$) gives rise to a stratification for $M$ for which the top-dimensional 
stratum $M[n]$ is open and dense in $M$. Note that $M[n-1]$ is empty, and thus $M$ is the metric completion of $M[n]$.

Now let us write the cone metric in terms of coordinates, at least locally. We take the flat metric with a cone singularity 
along a real codimension 2 submanifold for example. 
Write $\mathbb{R}^{n}=\mathbb{R}^{2}\times\mathbb{R}^{n-2}$ 
and let $D$ denote the codimension 2 submanifold $\{0\}\times\mathbb{R}^{n-2}$. Take polar coordinates $r,\theta$ on $\mathbb{R}^{2}$ 
and standard coordinates $w_{i}$ on $\mathbb{R}^{n-2}$. Fix $\beta\in (0,1)$ and the following singular metric 
\[
g=dr^{2}+\beta^{2}r^{2}d\theta^{2}+\sum dw_{i}^{2}
\]
is called the \emph{standard cone metric} with cone angle $2\pi\beta$ along $D$.

We can also take a complex point of view, let $\mathbb{C}_{\beta}$ denote the complex plane $\mathbb{C}$ endowed with the metric 
$\beta^{2}|\xi|^{2(\beta-1)}|d\xi|^{2}$, then there is an isometry from the above singular metric 
on $\mathbb{R}^{2m}$ to the Riemannian product $\mathbb{C}_{\beta}\times\mathbb{C}^{m-1}$, realized by taking the coordinate 
$\xi=r^{\frac{1}{\beta}}e^{\sqrt{-1}\theta}$ on $\mathbb{C}_{\beta}$ and taking standard complex coordinates on $\mathbb{C}^{m-1}$. 
This indeed describes the flat K\"{a}hler metric with a cone singularity along a smooth hypersurface. 
It also provides a local model for the tangent cone at a generic point of a divisor along which the concerned K\"{a}hler metric 
is of a conical type. For foundations of K\"{a}hler metrics with cone singularities along a divisor, interested readers can 
consult \cite{Donaldson}.

Next, we discuss the conical structure on a stratum of higher codimension in a given complex cone-manifold $M$, 
for which we assume that the stratum inherits a complex structure from $M$ (so that it is always of even real dimension). 
Let $K$ denote an irreducible stratum of complex codimension $k$. Then the normal slice to $K$ at a point $z\in K$ is a union of 
`complex rays' each of which is swept out by a real ray along a circle. The space of complex rays is the \emph{complex link} of 
the stratum, which is of a complex spherical cone-manifold structure of dimension $k-1$. Then a Seifert fiber space over the 
complex link with generic fiber a circle of length $\gamma(K)$ is the \emph{real link} of that stratum. 
We call the length $\gamma(K)$ the \emph{scalar cone angle} at $K$.

Under this setting, our ball quotient $X=\mathbb{B}/\Gamma$ is a complete $(\mathbb{CH}^{n},\mathrm{PU}(1,n))-$cone-manifold, 
where $\mathbb{CH}^{n}$ is the $n$-dimensional complex hyperbolic space and $\mathrm{PU}(1,n)$ (identified with 
$\mathrm{SU}(1,n)$ modulo its center) is its holomorphic isometry group. For simplicity it will also be called a 
\emph{(complex) hyperbolic cone-manifold} in what follows.

The collection $\mathscr{H}_{X^{J}}$ of the hypersurfaces $D_{i}$ forms an arrangement on $X^{J}$. 
Consider the log pair $(X^{J},\Delta^{J})$ where $\Delta^{J}=\sum a_{i}D_{i}$. Then the cone-manifold $X$ has a cone angle 
$2\pi(1-a_{i})$ at a generic point of the divisor $D_{i}$. The curvature model around that point in terms of local 
coordinates is given by 
\[
-\sqrt{-1}\partial\bar{\partial}\log \big(1-(|z_{1}|^{2(1-a_{i})}+\sum_{j=2}^{n}|z_{j}|^{2})\big).
\]

On a stratum $L^{\circ}$ of codimension $k>1$, assume that each $D_{i}$ is without self-intersection, define 
\[
a_{L}=(\mathrm{codim}L)^{-1}\sum_{i:D_{i}\supset L}a_{i}
\]
as in Remark $\ref{rem:exponents}$, then we have the following characterization for the scalar cone angle at $L^{\circ}$.
\begin{proposition}
The scalar cone angle at $L^{\circ}$ is $2\pi(1-a_{L})$.
\end{proposition}

\begin{proof}
Consider the tautological $\mathbb{C}^{\times}$-bundle $\mathscr{X}^{J}\rightarrow X^{J}$. By Lemma $\ref{lem:exponents}$, 
the affine structure on $\mathscr{X}^{\circ}$ is of infinitesimally simple type along the blow-up 
$\mathscr{E}(\mathscr{L})^{\circ}$ of $\mathscr{L}$ with logarithmic exponents $0$ and $a_{\mathscr{L}}-1$. 
In particular, the logarithmic exponent along the normal space $N(\mathscr{L},\mathscr{X}^{J})$ is 
$a_{\mathscr{L}}-1$. So the Lie derivative with respect to the dilatation field acts on flat vector fields 
on the normal space $N(\mathscr{L},\mathscr{X}^{J})$ simply as multiplication by $1-a_{\mathscr{L}}$ ($=1-a_{L}$). 
When we look at the projectivized space $X^{J}$, the action does not change on the normal space $N(L,X^{J})$. 
Then the scalar cone angle at $L^{\circ}$ is $2\pi(1-a_{L})$.
\end{proof}

\begin{remark}\label{rem:cone angle}
We notice that 
\[
a_{L}-1<0
\]
for any stratum $L^{\circ}$ on the ball quotient $X$ (without the cusps) since $2\pi(1-a_{L})$ is the scalar cone angle 
at $L^{\circ}$ on $X$ and it must be positive.

At a cusp $I$, it is clear that $a_{I}=1$, so we can view the angle at a cusp as $0$.
\end{remark}

\section{Singularities for GIT compactifications}\label{sec:GIT singularities}

In this section we discuss the MMP singularities for the GIT compactification, through the birational transformation 
from the Baily-Borel compactification, so that we can compare the MMP singularities on both sides.

According to Theorem \ref{thm:GIT compactifications}, we have that the arrangement complement $X^{\circ}$ can be 
extended to a compactification $\hat{X}^{\mathscr{H}}\cong Y^{ss}/\!\!/G$ which is of the GIT interpretation. 
If the relevant ball arrangement $\mathscr{H}$ is understood, we write $\widehat{X^{\circ}}$ for $\hat{X}^{\mathscr{H}}$, 
$\tilde{X}$ for $\tilde{X}^{\mathscr{H}}$ and $X^{J}$ for $X^{J_{\mathscr{H}}}$ to suppress $\mathscr{H}$. 
By the discussion of Section $\ref{sec:birational-operation}$, we have the following diagram
\[
\begin{tikzpicture}[scale=2]
	\node (A) at (0,0) {$\widehat{X^{\circ}}$};
	\node (B) at (1,1) {$\tilde{X}$};
	\node (C) at (1.5,0.5) {$X^{J}$};
	\node (D) at (2,0) {$X^{*}$};
	\draw
	(A) edge[<-,>=angle 60]      (B)
	(B) edge[->,>=angle 60]      (C)
	(C) edge[->,>=angle 60]      (D)
	(A) edge[<->,>=angle 60,dashed]      (D);
\end{tikzpicture}
\]
for which $\tilde{X}\rightarrow X^{*}$ is a log resolution of $(X^{*},\Delta^{*})$.

Now every $L\in\Lambda(\mathscr{H}_{X^{J}})$ determines an exceptional divisor $E(L)$ and all of these 
exceptional divisors form a normal crossing divisor in $\tilde{X}$. Since $\mathscr{H}_{X^{J},L}$ (resp. $\mathscr{H}_{X^{J}}^{L}$) 
defines an arrangement in $L$ (resp. $\mathbb{P}(L,X^{J})$), 
the divisor $E(L)$ can be identified with $\tilde{L}\times\tilde{\mathbb{P}}(L,X^{J})$. 
Moreover, the divisor $E(L)$ contains an open dense stratum $L^{\circ}\times\mathbb{P}(L,X^{J})^{\circ}$, denoted by $E(L)^{\circ}$. 
We already know that the affine structure on the $\mathbb{C}^{\times}$-bundle $\mathscr{X}^{\circ}$ of $X^{\circ}$ degenerates 
infinitesimally simply along $\mathscr{E}(\mathscr{L})^{\circ}$ (in the sense of Definition \ref{def:connection}) 
with logarithmic exponent $0$ and $a_{\mathscr{L}}-1$ ($=a_{L}-1$ by Remark \ref{rem:exponents}), 
so that the behavior of the projective developing map 
near $E(L)^{\circ}$ can be understood in an explicit manner.

Now assume $a_{L}-1\neq 0$, 
write a point $z=(z_{0},z_{1})\in E(L)^{\circ}$, let $\tilde{X}_{z}$ denote the germ of $\tilde{X}$ at $z$ and so for $L^{\circ}_{z_{0}}$, 
etc., then there exist a submersion $F_{0}:\tilde{X}_{z}\rightarrow L^{\circ}_{z_{0}}$, 
a submersion $F_{1}:\tilde{X}_{z}\rightarrow T_{1}$ 
(with $T_{1}$ a vector space, identified with $\mathbb{P}(L,X^{J})^{\circ}_{z_{1}}$) 
and a defining equation $t$ for $E(L)^{\circ}_{z}$ 
such that $(F_{0},t,F_{1})$ is an affine chart for $\tilde{X}_{z}$, and the developing map is projectively equivalent to the map 
\[
[(1,F_{0}),t^{1-a_{L}},t^{1-a_{L}}F_{1}]:\tilde{X}_{z}\rightarrow L^{\circ}_{z_{0}}\times \mathbb{C}\times T_{1}.
\]

Furthermore, the divisors $E(L)$ also determine a stratification of $\tilde{X}$. An arbitrary stratum can be described as follows: 
the intersection $\cap_{i}E(L_{i})$ is nonempty if and only if these members make up a flag: 
$L_{\bullet}:=L_{0}\subset L_{1}\subset \cdots \subset L_{k}\subset L_{k+1}=X^{J}$. Their intersection $E(L_{\bullet})$ 
can be identified with the product 
\[
E(L_{\bullet})=\tilde{L}_{0}\times\tilde{\mathbb{P}}(L_{0},L_{1})\times\cdots\times\tilde{\mathbb{P}}(L_{k},X^{J}),
\]
and it contains an open dense stratum $E(L_{\bullet})^{\circ}$ decomposed as follows 
\[
E(L_{\bullet})^{\circ}=L_{0}^{\circ}\times\mathbb{P}(L_{0},L_{1})^{\circ}\times\cdots\times\mathbb{P}(L_{k},X^{J})^{\circ}.
\]
Assume that $a_{L_{i}}-1\neq 0$ for each $L_{i}$. 
For a point $z=(z_{0},z_{1},\dots,z_{k+1})\in E(L_{\bullet})^{\circ}$, there exist a submersion 
$F_{0}:\tilde{X}_{z}\rightarrow L^{\circ}_{z_{0}}$, 
a submersion $F_{i}:\tilde{X}_{z}\rightarrow T_{i}$ (with $T_{i}$ a vector space, identified with 
$\mathbb{P}(L_{i-1},L_{i})^{\circ}_{z_{i}}$) and a defining equation $t_{i-1}$ of $E(L_{i-1})$ for each $i=1,\dots,k+1$ 
such that $\big(F_{0},(t_{i-1},F_{i})_{i=1}^{k+1}\big)$ is 
a chart for $\tilde{X}_{z}$, and the developing map is projectively equivalent to the map
\[
\Big[(1,F_{0}),\big(t_{0}^{1-a_{L_{0}}}\cdots t_{i-1}^{1-a_{L_{i-1}}}(1,F_{i})\big)_{i=1}^{k+1}\Big]:\tilde{X}_{z}\rightarrow 
L^{\circ}_{z_{0}}\times \prod_{i=1}^{k+1}(\mathbb{C}\times T_{i}).
\]

Therefore, the projective developing map near $E(L)^{\circ}_{z}$ is equivalent to 
\begin{align}\label{eqn:map1}
[(1,F_{0}),t^{1-a_{L}},t^{1-a_{L}}F_{1}]:\tilde{X}_{z}\rightarrow L^{\circ}_{z_{0}}\times \mathbb{C}\times T_{1},
\end{align}
which extends across $\tilde{X}_{z}$ with restriction to $E(L)^{\circ}_{z}$ essentially given by $[1,F_{0}]$. 
This is realized by the projection $E(L)^{\circ}_{z}\rightarrow L^{\circ}_{z_{0}}$. Likewise, the projective developing map 
near $E(L_{\bullet})^{\circ}_{z}$ is equivalent to 
\begin{align}\label{eqn:map2}
\Big[(1,F_{0}),\big(t_{0}^{1-a_{L_{0}}}\cdots t_{i-1}^{1-a_{L_{i-1}}}(1,F_{i})\big)_{i=1}^{k+1}\Big]:\tilde{X}_{z}\rightarrow 
L^{\circ}_{z_{0}}\times \prod_{i=1}^{k+1}(\mathbb{C}\times T_{i}),
\end{align}
which extends across $\tilde{X}_{z}$ with restriction to $E(L_{\bullet})^{\circ}_{z}$ essentially given by $[1,F_{0}]$. 
This is realized by the projection $E(L_{\bullet})^{\circ}_{z}\rightarrow L^{\circ}_{z_{0}}$. 
So that we get $X^{J}$ from $\tilde{X}$.

On the other hand, the projective developing map ($\ref{eqn:map1}$) near $E(L)^{\circ}_{z}$ can also be written as 
\[
[t^{a_{L}-1}(1,F_{0}),1,F_{1}]:\tilde{X}_{z}\rightarrow L^{\circ}_{z_{0}}\times \mathbb{C}\times T_{1}.
\]
which extends across $\tilde{X}_{z}$ with restriction to $E(L)^{\circ}_{z}$ essentially given by $[1,F_{1}]$. 
This is realized by the projection $E(L)^{\circ}_{z}\rightarrow \mathbb{P}(L,X^{J})^{\circ}_{z_{1}}$. 
Likewise, the projective developing map ($\ref{eqn:map2}$) near $E(L_{\bullet})^{\circ}_{z}$ can be written as 
\[
\Big[\big(t_{i}^{a_{L_{i}}-1}\cdots t_{k}^{a_{L_{k}}-1}(1,F_{i})\big)_{i=0}^{k},(1,F_{k+1})\Big]:\tilde{X}_{z}\rightarrow 
L^{\circ}_{z_{0}}\times \prod_{i=1}^{k+1}(\mathbb{C}\times T_{i}),
\]
which extends across $\tilde{X}_{z}$ with restriction to $E(L_{\bullet})^{\circ}_{z}$ essentially given by $[1,F_{k+1}]$. 
This is realized by the projection $E(L_{\bullet})^{\circ}_{z}\rightarrow \mathbb{P}(L_{k},X^{J})^{\circ}_{z_{k+1}}$. 
So that we get $\widehat{X^{\circ}}$ from $\tilde{X}$.

The following proposition relates the logarithmic exponents to the discrepancy. It is a straightforward generalization 
of the smooth case \cite[Exercise II.8.5]{Hartshorne} to the fractional case.

\begin{proposition}\label{prop:discrepancy}
Let $X$ be a normal $\mathbb{Q}$-Gorenstein variety. Denote by $X_{sm}$ the smooth locus of $X$. 
Let $Z\subset X_{sm}$ be a smooth subvariety of codimension $k\geq 2$. Let $f:Bl_{Z}X\rightarrow X$ be the blow-up of $Z$ 
and $E\subset Bl_{Z}X$ the $f$-exceptional divisor. 
Assume that the projective structure on $Bl_{Z}-E$ (descended from the affine structure on a $\mathbb{C}^{\times}$-bundle 
of $Bl_{Z}-E$ which respects the scalar multiplication) is of simple type along $E$ with relative logarithmic exponent 
$\lambda$ ($\in\mathbb{Q}\cap(-1,0)$). 
Then the discrepancy 
\[
a(E,X)=-k\lambda-1.
\]
\end{proposition}

\begin{proof}
Since $X$ is normal and $\mathbb{Q}$-Gorenstein, we have 
\[
K_{Y}=f^{*}K_{X}+a(E,X)E
\]
where $Y=Bl_{Z}X$, we want to determine the discrepancy $a(E,X)$. For that we write 
\[
\omega_{Y}=f^{*}\omega_{X}\otimes\mathscr{O}_{Y}(aE)
\]
for some rational number $a=a(E,X)$. By adjunction, we have 
\begin{align*}
\omega_{E}&=\omega_{Y}\otimes\mathscr{O}_{Y}(E)\otimes\mathscr{O}_{E} \\
&=f^{*}\omega_{X}\otimes\mathscr{O}_{Y}((a+1)E)\otimes\mathscr{O}_{E} \\
&=f^{*}\omega_{X}\otimes\mathscr{I}_{E}^{\otimes(-a-1)}\otimes\mathscr{O}_{E} \\
&=f^{*}\omega_{X}\otimes\mathscr{O}_{Y}(1)^{\otimes(-a-1)}\otimes\mathscr{O}_{E} \\
&=f^{*}\omega_{X}\otimes\mathscr{O}_{E}(-a-1).
\end{align*}
Now take a closed point of $z\in Z$ and let $W$ be the fiber of $E$ over $z$, i.e., 
$W=z\times_{Z}E$. Denote the respective projections by $p_{1}$ and $p_{2}$. Then we have 
\begin{align*}
\omega_{W}&=p_{1}^{*}\omega_{z}\otimes p_{2}^{*}\omega_{E} \\
&=p_{1}^{*}\mathscr{O}_{z}\otimes p_{2}^{*}(f^{*}\omega_{X}\otimes\mathscr{O}_{E}(-a-1)) \\
&=p_{2}^{*}(f^{*}\omega_{X}\otimes\mathscr{O}_{E}(-a-1)) \\
&=\mathscr{O}_{W}\otimes p_{2}^{*}\mathscr{O}_{E}(-a-1) \\
&=\mathscr{O}_{W}(-a-1).
\end{align*}
Let us look at the projective structure near $W$, since we assume that the projective structure on $Y-E$ is of 
simple type along $E$ with relative logarithmic exponent $\lambda$, then the developing map near a point $z$ of $W$ 
restricted to $W$ is projectively equivalent to 
\[
[F_{0},t^{-\lambda},t^{-\lambda}F_{\lambda}]|_{W}=[t^{-\lambda},t^{-\lambda}F_{\lambda}].
\]
This shows that $W$ is just a $\mathbb{P}^{k-1}$ acted by a dilatation field with factor $\lambda$. 
So we have 
\[
\omega_{W}=\mathscr{O}_{W}(-k)^{\otimes (-\lambda)}=\mathscr{O}_{W}(k\lambda).
\]
Thus we have 
\[
a(E,X)=a=-k\lambda-1.
\]
\end{proof}

We know that the difference $\widehat{X^{\circ}}-X^{\circ}$ consists of those strata $\mathbb{P}(L,X^{J})^{\circ}$ 
associated to $L\in\Lambda(\mathscr{H}_{X^{J}})$. 
We have that $\widehat{X^{\circ}}=Y^{ss}/\!\!/G$ is normal since we assume $Y$ is normal. 
We notice that $\widehat{X^{\circ}}-X^{\circ}$ is a birational transform of the arrangement $\mathscr{H}_{X}$ plus 
a partial resolution of the cusps and also assume that it is of codimension $\geq 2$, it implies that 
$K_{\widehat{X^{\circ}}}$ is $\mathbb{Q}$-Cartier. It is clear that $\widehat{X^{\circ}}$ inherits a projective structure 
from $\tilde{X}$.

We next show what type of the MMP singularities of the GIT compactification $\widehat{X^{\circ}}$ are.

\begin{theorem}\label{thm:GIT non-lc}
Suppose that the GIT compactification $\widehat{X^{\circ}}$ is a birational modification of the Baily-Borel 
compactification $X^{*}$ in terms of an arrangement, or equivalently, the relevant 
arrangement is nonempty, then the GIT compactification $\widehat{X^{\circ}}$ has non-lc singularities.
\end{theorem}

\begin{proof}
Since the relevant arrangement is not empty, there exists a stratum $Z\subset\widehat{X^{\circ}}$ which comes from 
the contraction $E(L)^{\circ}\rightarrow\mathbb{P}(L,X^{J})^{\circ}$ for some $L\in\Lambda(\mathscr{H}_{X^{J}})$ 
with $1-a_{L}\neq 0$, namely $Z=\mathbb{P}(L,X^{J})^{\circ}$. 
As discussed above, near a generic point $z=(z_{0},z_{1})$ on $E(L)^{\circ}\subset\tilde{X}$ the projective developing map 
is given by 
\[
[(1,F_{0}),t^{1-a_{L}},t^{1-a_{L}}F_{1}]:\tilde{X}_{z}\rightarrow L^{\circ}_{z_{0}}\times \mathbb{C}\times T_{1},
\]
with $T_{1}$ a vector space, identified with $\mathbb{P}(L,X^{J})^{\circ}_{z_{1}}$. 
On the other hand, if we regard the exceptional divisor $E(L)^{\circ}$ as the blow-up of 
$Z=\mathbb{P}(L,X^{J})^{\circ}$ from $\widehat{X^{\circ}}$, the above projective developing map can also be written as 
\[
[t^{a_{L}-1}(1,F_{0}),1,F_{1}]:\tilde{X}_{z}\rightarrow L^{\circ}_{z_{0}}\times \mathbb{C}\times T_{1}.
\]

Let $k:=\mathrm{codim}Z$. Notice that now the logarithmic exponent $\lambda$ equals to $-(a_{L}-1)$. 
By Proposition $\ref{prop:discrepancy}$ we can calculate the discrepancy 
of the exceptional divisor $E(Z)$ of $Z$ for the variety $\widehat{X^{\circ}}$: 
\begin{align*}
	a(E(Z),\widehat{X^{\circ}})&=-k\lambda-1 \\
	&=k(a_{L}-1)-1 \\
	&<-1,
\end{align*}
since $a_{L}-1<0$ by Remark $\ref{rem:cone angle}$. This proves the theorem. 
\end{proof}

So we have the following: 
\begin{corollary}\label{cor:non-lc to lc}
The birational transformation $\widehat{X^{\circ}}\dashrightarrow X^{*}$ turns non-lc singularities on $\widehat{X^{\circ}}$ 
to lc singularities on $X^{*}$.
\end{corollary}

\section{Applications}\label{sec:applications}

We illustrate the preceding results by discussing three examples in this section: the moduli spaces of quartic curves, 
rational elliptic surfaces and cubic threefolds.

Now let us consider the situation when the imaginary quadratic field $k$ are the cyclotomic ones, that is, 
$\mathbb{Q}(\zeta_{4})$ and $\mathbb{Q}(\zeta_{6})$. Their corresponding rings of integers $\mathcal{O}_{k}$ are thus the 
Gaussian integers $\mathbb{Z}[\zeta_{4}]$ and the Eisenstein integers $\mathbb{Z}[\zeta_{6}]$, respectively. 
For any tree directed graph $I_{n}$, i.e., a finite graph with $n$ numbered nodes without loops or multiple edges and 
each edge is directed, we can associate a Hermitian lattice, denoted by $\mathbb{Z}[\zeta_{m}](I_{n})$ for $m=4$ or $6$, in 
such a way that it is the free $\mathbb{Z}[\zeta_{m}]$-module living on the set $\{\alpha_{1},\dots,\alpha_{n}\}$ indexed by 
the nodes together with a $\mathbb{Z}[\zeta_{m}]$-valued Hermitian form $h$ defined by 
\begin{equation*}
h(\alpha_{p},\alpha_{q})=
\left\{
\begin{aligned}
&|1+\zeta_{m}|^{2}=m/2  && \text{if $p=q$,} \\
&1+\zeta_{m} \; \text{(for m=4)}; \; \zeta_{m}-\bar{\zeta_{m}} \; \text{(for m=6)} && \text{if $(p,q)$ is a directed edge,} \\
&0  && \text{if $p$ and $q$ are not connected}.
\end{aligned}
\right.
\end{equation*}
Up to isomorphism this lattice $\mathbb{Z}[\zeta_{m}](I_{n})$ is independent of the directions. 
Then the group of unitary transformations $\mathrm{U}(\mathbb{Z}[\zeta_{m}](I_{n}))$ of the lattice 
forms an arithmetic group in the unitary group $\mathrm{U}(\mathbb{Z}[\zeta_{m}](I_{n})\otimes\mathbb{C})$ of the complexification 
of $\mathbb{Z}[\zeta_{m}](I_{n})$. Those elements in $\mathbb{Z}[\zeta_{m}](I_{n})$ with square norm $m/2$ are called 
the \emph{roots} of $\mathbb{Z}[\zeta_{m}](I_{n})$.

\subsection{The moduli space of quartic curves}

Let us look at the first example, the moduli space of quartic curves. The invariant theory of quartic curves is classical 
due to Hilbert-Mumford. Let $Y$ be the space of ternary forms of degree $4$ up to scalar, i.e., 
$Y:=\mathbb{P}(H^{0}(\mathbb{P}^{2},\mathscr{O}_{\mathbb{P}^{2}}(4)))$, then let $\eta:=\mathscr{O}_{Y}(1)$ 
and $G:=\mathrm{SL}(3,\mathbb{C})$. According to \cite{Mumford-Fogarty-Kirwan}, a point of $Y$ is $G$-stable precisely if its 
corresponding curve is reduced and with possible singularities cusps at worst, and is $G$-semistable precisely if it is 
reduced and with possible singularities tacnodes at worst or is a nonsingular conic of multiplicity $2$. The minimal strictly 
semistable curves are the members of a pencil $f\cdot(\lambda f+(1-\lambda)g)$ with $f=0$ a nonsingular conic and $g=0$ two 
distinct lines tangent to the conic $f=0$. The generic member of the pencil is the sum of two nonsingular conics intersecting at 
two tacnodes. There are also two possible strictly semistable degenerations, one is a nonsingular conic with two distinct lines 
tangent to it for $\lambda=0$, the other one is a nonsingular conic of multiplicity $2$ for $\lambda=1$. Such pairs of conics 
have modulus $1$, so $Y^{ss}/\!\!/G-Y^{st}/G$ consists of a curve. The nonsingular quartics are stable and define an open subset 
$Y^{\diamond}\subset Y$.

A period map taking values in a ball quotient was obtained by Kond\={o} \cite{Kondo}. For a smooth quartic curve 
$Q\subset\mathbb{P}^{2}$ in a ternary form $f(x,y,z)$ of degree $4$, consider the surfaces 
\[
S:w^{2}=f(x,y,z);\qquad K:w^{4}=f(x,y,z)
\]
totally ramified along $Q$ of order $2$ and $4$ respectively. Then the surface $K$ is a smooth $K3$ surface with an automorphism 
$\sigma$ of order $4$. It doubly covers the intermediate degree $2$ del Pezzo surface $S$. And the anticanonical map of $S$ realizes 
it as a double cover of $\mathbb{P}^{2}$ ramified along $Q$. This identifies the coarse moduli space of degree $2$ del Pezzo surfaces 
with the coarse moduli space of smooth quartic curves. It is known that the algebraic K3 surfaces have a unique up to scalar 
holomorphc $2$-form, taking their periods leads to a period map $Per:\mathcal{K}\rightarrow \mathbb{D}/\Gamma_{\mathbb{D}}$ with 
$\mathbb{D}$ a type IV domain of dimension $19$ and $\Gamma_{\mathbb{D}}$ a corresponding arithmetic group. It is clear that 
the assignment $Q\mapsto K$ gives rise to an injection $\mathcal{Q}\rightarrow\mathcal{K}$ from the moduli space $\mathcal{Q}$ 
of smooth quartic curves to the moduli space $\mathcal{K}$ of algebraic K3 surfaces. Kond\={o} has shown that 
the restriction of the period map to $\mathcal{Q}$ lands in a $6$-dimensional ball quotient $X$, the one associated to the 
Hermitian lattice $\mathbb{Z}[\zeta_{4}](E_{7})$ of type $E_{7}$, namely, $X:=\mathbb{B}/\mathrm{U}(\mathbb{Z}[\zeta_{4}](E_{7}))$. 
This defines the period map for the moduli space of smooth quartic curves 
\[
Per:Y^{\diamond}/G=\mathcal{Q}\rightarrow X,
\]
which gives an isomorphism of $\mathcal{Q}$ onto the complement of a divisor $D$ in the ball quotient $X$.

It turns out that $\mathrm{U}(\mathbb{Z}[\zeta_{4}](E_{7}))$ acts transitively on the set of cusps so that $X^{*}$ is topologically 
the $1$-point compactification of $X$. The divisor $D$ comes from the full arrangement $\mathscr{H}$ on $\mathbb{B}$ 
which is the union of hyperplanes orthogonal to a root. They come in 
two orbits under the action of $\mathrm{U}(\mathbb{Z}[\zeta_{4}](E_{7}))$: $\mathscr{H}=\mathscr{H}_{n}\cup\mathscr{H}_{h}$, 
which gives rise to two irreducible components $D_{n}$ and $D_{h}$ of $D$ in $X$. A generic point of $D_{n}$ parametrizes a plane 
quartic curve with a node, and a generic point of $D_{h}$ parametrizes a hyperelliptic curve of genus $3$. The hyperplanes 
of type $\mathscr{H}_{h}$ do not meet inside $\mathbb{B}$ so that $D_{h}$ is not self-intersected. The period map 
extends to an isomorphism
\[
Y^{st}/G\xrightarrow{\cong} X-D_{h},
\] 
on which there is an isomorphism of orbifold line bundles $\eta/G\cong\mathscr{L}|(X-D_{h})$. Applying Theorem 
$\ref{thm:GIT compactifications}$ to $U=Y^{st}$, and Corollary $\ref{cor:non-lc to lc}$ tells us the following fact. 
\begin{corollary}
The period map extends to an isomorphism 
\[
Y^{ss}/\!\!/G\xrightarrow{\cong} \hat{X}^{\mathscr{H}_{h}}(=:\widehat{X^{\circ}}),
\]
and the birational transformation from $\widehat{X^{\circ}}$ to $X^{*}$ turns non-lc singularities to lc singularities.
\end{corollary}
From $X^{*}$ to $\widehat{X^{\circ}}$, it looks like a one-dimensional blow-up of the single cusp in $X^{*}$ 
to the strictly semistable stratum, and then with the divisor $D_{h}$ contracted.

\subsection{The moduli space of rational elliptic surfaces}

By a \emph{rational elliptic surface} we mean a rational surface which admits a relatively minimal elliptic fibration with a section. 
Its anticanonical map is base point free and realizes the surface as a fibration onto $\mathbb{P}^{1}$. So a rational elliptic 
surface can always be written in a Weierstrass form: $y^{2}=x^{3}+3a(t)x+2b(t)$, where $x$, $y$, $t$ are the affine coordinates 
of a $\mathbb{P}^{2}$-bundle over $\mathbb{P}^{1}$, and $a(t)$ and $b(t)$ are rational functions of degree $4$ and $6$, 
respectively. A compactification of its moduli space using geometric invariant theory was constructed by Miranda in \cite{Miranda}. 

Put $H_{m}:=H^{0}(\mathbb{P}^{1},\mathscr{O}_{\mathbb{P}^{1}}(m))$. Then $P_{m}:=\mathbb{P}(H_{m})$ is the space of effective degree 
$m$ divisors (of dimension $m$). Let $Y$ be the weighted projective space obtained by dividing $H_{4}\oplus H_{6}-\{(0,0)\}$ out by 
the action of the center $\mathbb{C}^{\times}\subset\mathrm{GL}(2,\mathbb{C})$, and let $G:=\mathrm{SL}(2,\mathbb{C})$. By assigning 
to a fiber of a rational elliptic surface its Euler number we have a so-called \emph{discriminant divisor} on its base curve, 
which is effective and of degree $12$. Then we have an equivariant discriminant morphism
\[
d:Y\rightarrow P_{12},\qquad (a,b)\mapsto a^{3}+b^{2}
\]
which is finite and birational to its image, a hypersurface in $P_{12}$. There is a natural ample orbifold line bundle $\eta$ over 
$Y$ endowed with an action of $G$ such that the pullback $d^{*}(\mathscr{O}_{P_{12}}(1))$ of $\mathscr{O}_{P_{12}}(1)$ 
under the morphism $d$ is equivariantly isomorphic to $\eta^{\otimes 6}$. In general the discriminant divisor is reduced 
for which it is a complete invariant (up to projective equivalence) to determine the isomorphism classes of the surfaces. 
These reduced discriminant divisors form a subset in $P_{12}$, denoted by $P_{12}^{\diamond}$. 
Then the preimage $Y^{\diamond}:=d^{-1}(P_{12}^{\diamond})$ parametrizes those rational elliptic surfaces with 
reduced discriminant and is isomorphic to its image in $P_{12}$. It is clearly contained in the stable set $Y^{st}$. However, the set 
$Y^{st}$ does not coincide with the preimage $d^{-1}(P_{12}^{st})$ of the stable set $P_{12}^{st}$. Recall that, 
following Hilbert, a degree $12$ divisor on $\mathbb{P}^{1}$ is semistable precisely if its all multiplicities are 
$\leq \frac{1}{2}\cdot 12=6$, and there is only one minimal strictly semistable orbit whose elements are represented by a 
divisor with two distinct points of multiplicities $6$. We denote the $G$-orbit spaces of $Y^{\diamond}$ and 
$Y^{st}\cap d^{-1}(P_{12}^{st})$ by $\mathcal{M}$ and $\mathcal{M}^{\sim}$ respectively. It can be shown that 
the space  $\mathcal{M}^{\sim}$ parametrizes the so-called Miranda-stable rational elliptic surfaces with reduced fibers and stable 
discriminant for which the allowed singular fibers are of Kodaira type $I_{r}$ with $r<6$, $II$, $III$ or $IV$, whereas 
the difference $Y^{st}/G-\mathcal{M}^{\sim}$ parametrizes rational elliptic surfaces with a fiber of type $I_{r}$ with 
$6\leq r\leq 9$. Furthermore, the difference $Y^{ss}/\!\!/G-Y^{st}/G$ parametrizes rational elliptic surfaces with two fibers 
of type $I_{0}^{*}$ or a fiber of type $I_{4}^{*}$, which form a rational curve. We call the projective variety 
$\mathcal{M}^{M}:=Y^{ss}/\!\!/G$ the \emph{Miranda compactification}. The boundary 
$\mathcal{M}^{M}-\mathcal{M}^{\sim}$ is of codimension $5$ in $\mathcal{M}^{M}$ and $\mathcal{M}^{M}-\mathcal{M}$ is 
of codimension $1$ in $\mathcal{M}^{M}$.

The period map for rational elliptic surfaces can be defined through the $G$-orbit space of $P_{12}^{\diamond}$ by the 
Deligne-Mostow theory \cite{Deligne-Mostow}. The space $P_{12}^{\diamond}/G$ parametrizes the unordered tuples of 
$12$ distinct marked points on $\mathbb{P}^{1}$ up to the projective equivalence. The Deligne-Mostow theory has 
shown that this space can be biholomorphically mapped onto a divisor complement of a $9$-dimensional ball quotient, 
the one associated to the Hermitan lattice $\mathbb{Z}[\zeta_{6}](A_{10})$ of type $A_{10}$, namely, 
$X_{o}:=\mathbb{B}_{o}/\mathrm{U}(\mathbb{Z}[\zeta_{6}](A_{10}))$. The group $\mathrm{U}(\mathbb{Z}[\zeta_{6}](A_{10}))$ 
acts on the cusps transitively so that the Baily-Borel compactification $X_{o}^{*}$ is topologically 
$1$-point compactification of $X_{o}$. The full arrangement $\mathscr{H}$ consisting of the hyperplanes perpendicular to a root 
are transitively permuted by the group $\mathrm{U}(\mathbb{Z}[\zeta_{6}](A_{10}))$, so they define an irreducible 
divisor $D_{o}$ in $X_{o}$. This gives rise to a period map 
\[
P_{12}^{\diamond}/G\rightarrow X_{o}-D_{o},
\]
which is an isomorphism. And it extends to isomorphisms
\[
P_{12}^{st}/G\xrightarrow{\cong} X_{o},\qquad P_{12}^{ss}/\!\!/G\xrightarrow{\cong} X_{o}^{*}.
\]
The latter identifies $\mathscr{L}^{\otimes 12}$ with $\mathcal{O}_{P_{12}^{ss}}(1)/\mathrm{U}(\mathbb{Z}[\zeta_{6}](A_{10}))$.

Then the discriminant morphism $d$ induces a period map for the moduli space of rational elliptic surfaces with 
reduced discriminant 
\[
Per:\mathcal{M}=Y^{\diamond}/G\rightarrow P_{12}^{\diamond}/G\rightarrow X_{o}-D_{o},
\]
and the pullback of $\mathscr{L}^{\otimes 12}$ can be identified with $\eta^{\otimes 6}/G$. Heckman and Looijenga have 
carefully analyzed this situation in \cite{Heckman-Looijenga}. It has been shown that this morphism is a closed 
embedding with a hyperball quotient as its image. This hyperball can be given as follows: choose a vector in 
$\mathbb{Z}[\zeta_{6}](A_{10})$ of square norm $6$ (can be realized as the sum of two perpendicular roots), 
and consider its orthogonal complement in $\mathbb{Z}[\zeta_{6}](A_{10})$, which we denote by $\Phi$. 
It has been shown that all such vectors fall into a single orbit under the action of $\mathrm{U}(\mathbb{Z}[\zeta_{6}](A_{10}))$ 
so that the choice of the vector is irrelevant. The $\mathrm{U}(\mathbb{Z}[\zeta_{6}](A_{10}))$-stabilizer of 
$\Phi$ acts on $\Phi$ through its full unitary group, which we denote by $\Gamma$. The sublattice 
$\Phi\subset\mathbb{Z}[\zeta_{6}](A_{10})$ hence defines a hyperball $\mathbb{B}\subset\mathbb{B}_{o}$, 
a corresponding ball quotient $X:=\mathbb{B}/\Gamma$, and also a natural map $j:X\rightarrow X_{o}$. 
The restriction of the arrangement $\mathscr{H}$ to the hyperball $\mathbb{B}$ gives an arithmetically defined 
arrangement on $\mathbb{B}$, and the corresponding divisor $D$ on $X$ is just $j^{-1}(D_{o})$. Therefore, the 
above period map defines an isomorphism 
\[
Per:\mathcal{M}=Y^{\diamond}/G\xrightarrow{\cong} X-D.
\]
It can be easily extended to a morphism $\mathcal{M}^{\sim}\rightarrow X$. But it is not a surjection. That 
is because unlike the situation of the group $\mathrm{U}(\mathbb{Z}[\zeta_{6}](A_{10}))$ acting on the arrangement 
$\mathscr{H}$ which gives only one orbit, the restriction of $\mathscr{H}$ to the hyperball decomposes into 
four classes under the action of $\Gamma$. These four classes can be distinguished by a numerical invariant 
(refer to \cite{Heckman-Looijenga} for a precise description) with possible values $6$, $9$, $15$, $18$. 
So the restriction of $\mathscr{H}$ to the hyperball can be written as the disjoint union of $\Gamma$-equivalence 
classes $\mathscr{H}_{r}$ with $r$ running over these four numbers, and $\mathscr{H}_{r}$ hence defines an irreducible 
divisor $D_{r}$ in $X$. The period map on $\mathcal{M}^{\sim}$ will not take values in $D_{6}\cup D_{9}$. 
So the period map gives an isomorphism 
\[
Per:\mathcal{M}^{\sim}=(Y^{st}\cap d^{-1}(P_{12}^{st}))/G\xrightarrow{\cong} X-(D_{6}\cup D_{9}),
\]
where a generic point of $D_{15}$ (resp. $D_{18}$) parametrizes a rational elliptic surface with a fiber of type 
$II$ (resp. $I_{2}$). Over this isomorphism there is also an isomorphism of orbifold line bundles between $\eta^{\otimes 6}/G$ 
and the pullback of $\mathscr{L}^{\otimes 12}$. It has been shown in \cite{Heckman-Looijenga} that any nonempty 
intersection of members taken from $\mathscr{H}_{6}\cup\mathscr{H}_{9}$ with $\mathbb{B}$ has dimension at least 
$5$. We also have that the boundary $\mathcal{M}^{M}-\mathcal{M}^{\sim}$ is of codimension $5$ in $\mathcal{M}^{M}$. 
So that we can apply Theorem $\ref{thm:GIT compactifications}$ to $U=Y^{st}\cap d^{-1}(P_{12}^{st})$, 
and Corollary $\ref{cor:non-lc to lc}$ tells us the following. 
\begin{corollary}
The period map extends to an isomorphism 
\[
\mathcal{M}^{M}\xrightarrow{\cong}\hat{X}^{\mathscr{H}_{6}\cup\mathscr{H}_{9}}(=:\widehat{X^{\circ}}),
\]
for which the latter is the modification of $X^{*}$ defined by the arrangement $\mathscr{H}_{6}\cup\mathscr{H}_{9}$. 
The birational transformation from $\widehat{X^{\circ}}$ to $X^{*}$ turns non-lc singularities to lc singularities. 
\end{corollary}
From $X^{*}$ to $\widehat{X^{\circ}}$, it looks like a contraction of the divisors $D_{6}$ and $D_{9}$, 
and with their intersection flipped.

\subsection{The moduli space of cubic threefolds}

There is one more example, the moduli space of cubic threefolds, for which the situation is quite like that of the moduli 
space of quartic curves. Let $Y:=\mathbb{P}(H^{0}(\mathbb{P}^{4},\mathscr{O}_{\mathbb{P}^{4}}(3)))$, 
$\eta:=\mathscr{O}_{Y}(1)$ and $G:=\mathrm{SL}(5,\mathbb{C})$. The GIT for cubic threefolds was analyzed independently 
by Allcock \cite{Allcock} and Yokoyama \cite{Yokoyama}. They found that a point of $Y$ is $G$-stable if and only if its 
singularities are of type $A_{1}$, $A_{2}$, $A_{3}$ or $A_{4}$, and is $G$-semistable if and only if it has singularities 
$A_{k}$ or $D_{4}$ at worst or is the chordal cubic (i.e., the secant variety of a rational normal curve in $\mathbb{P}^{4}$ 
of degree $4$). The minimal strictly semistable cubic threefolds are the members of a pencil 
and the member corresponding to a special point. The pencil is of the form 
$f_{a,b}=ax_{2}^{3}+x_{0}x_{3}^{2}+x_{1}^{2}x_{4}-x_{0}x_{2}x_{4}+bx_{1}x_{2}x_{3}$, where we let $\lambda=\frac{4a}{b^{2}}$. 
The generic member (when $\lambda\neq 0,1$) of the pencil has precisely 2 singularities of type $A_{5}$. 
There are also two strictly semistable degenerations, one acquires an additional singularity of type $A_{1}$ (when $\lambda=0$), 
the other one becomes the chordal cubic (when $\lambda=1$). The special point corresponds to the orbit of 
the threefold given by $g=x_{0}x_{1}x_{2}+x_{3}^{3}+x_{4}^{3}$. This threefold has precisely 3 singularities of type 
$D_{4}$. So $Y^{ss}/\!\!/G-Y^{st}/G$ consists of a rational curve and an isolated point. 
The nonsingular threefolds are stable and define an open subset $Y^{\diamond}\subset Y$.

A period map taking values in a ball quotient was obtained by Allcock-Carlson-Toledo \cite{Allcock-Carlson-Toledo} and 
Looijenga-Swierstra \cite{Looijenga-Swierstra} (see also \cite{Casalaina-Martin--Laza}). It is studied by passing to 
the periods of cubic fourfolds obtained as the triple cover of $\mathbb{P}^{4}$ ramified along a cubic threefold. 
Namely, for a smooth cubic threefold $Q\subset\mathbb{P}^{4}$ in a form $f(x_{0},x_{1},x_{2},x_{3},x_{4})$ of degree $3$, 
we consider the cubic fourfold 
\[
K:x_{5}^{3}=f(x_{0},x_{1},x_{2},x_{3},x_{4})
\]
totally ramified along $Q$ of order $3$. Using Voisin's global Torelli theorem for cubic fourfolds, it can be shown that 
the moduli space of cubic threefolds lands in a 10-dimensional ball quotient $X$, the one associated to the Hermitian 
lattice $\Lambda$, namely, $X:=\mathbb{B}/\mathrm{U}(\Lambda)$. Here the lattice $\Lambda$ is given by 
$\Lambda:=(3)\oplus\mathbb{Z}[\zeta_{6}](A_{4})\oplus\mathbb{Z}[\zeta_{6}](A_{4})\oplus U_{\mathbb{Z}[\zeta_{6}]}$ 
where $U_{\mathbb{Z}[\zeta_{6}]}$ is the hyperbolic Eisenstein lattice (i.e., spanned by two isotropic vectors with 
inner product $\sqrt{-3}$). This defines the period map for the moduli space of smooth cubic threefolds 
\[
Per:Y^{\diamond}/G\rightarrow X,
\]
which gives an isomorphism of $Y^{\diamond}/G$ onto the complement of a divisor $D$ in the ball quotient $X$.

It turns out that $\mathrm{U}(\Lambda)$ has two orbits of cusps, corresponding to the degenerations to cubic threefolds 
with $A_{5}$ and $D_{4}$ singularities respectively. The divisor $D$ comes from the full arrangement $\mathscr{H}$ on 
$\mathbb{B}$ which is the union of hyperplanes orthogonal to a root. 
They come in two orbits under the action of $\mathrm{U}(\Lambda)$: $\mathscr{H}=\mathscr{H}_{c}\cup\mathscr{H}_{\Delta}$, 
which gives rise to two irreducible components $D_{c}$ and $D_{\Delta}$ of $D$ in $X$. We call $D_{c}$ and $D_{\Delta}$ 
the \emph{hyperelliptic divisor} and the \emph{discriminant divisor} respectively. The hyperplane sections of $\mathbb{B}$ 
of type $\mathscr{H}_{c}$ are disjoint so that $D_{c}$ is not self-intersected. The period map 
extends to an isomorphism
\[
Y^{st}/G\xrightarrow{\cong} X-D_{c},
\] 
on which there is an isomorphism of orbifold line bundles $\eta/G\cong\mathscr{L}|(X-D_{c})$. Applying Theorem 
$\ref{thm:GIT compactifications}$ to $U=Y^{st}$, and Corollary $\ref{cor:non-lc to lc}$ tells us the following fact. 
\begin{corollary}
	The period map extends to an isomorphism 
	\[
	Y^{ss}/\!\!/G\xrightarrow{\cong} \hat{X}^{\mathscr{H}_{c}}(=:\widehat{X^{\circ}}),
	\]
	and the birational transformation from $\widehat{X^{\circ}}$ to $X^{*}$ turns non-lc singularities to lc singularities.
\end{corollary}
From $X^{*}$ to $\widehat{X^{\circ}}$, it looks like a one-dimensional blow-up of one cusp in $X^{*}$ 
to the strictly semistable stratum that corresponds to the cubic threefolds with $A_{5}$ singularities, 
and then with the divisor $D_{c}$ contracted.

\bibliographystyle{alpha}
\bibliography{bibfile}
\end{document}